\documentclass{amsart}
\usepackage{graphicx}
\usepackage{enumerate,amssymb,amsmath,graphics,epsfig}
\usepackage{color,stmaryrd}
\usepackage{xspace}
\usepackage{accents}
\usepackage{xcolor, graphicx, subfigure, float}

\newcommand\RE{\mathbb{R}}
\renewcommand\div{\mathop{\rm{div}}\nolimits}

\newcommand\Grad{\mathop{\boldsymbol\nabla}\nolimits}

\newcommand\grad{\mathop\nabla\nolimits}
\newcommand\Huo{H^1_0(\Omega)^d}
\newcommand\Hub{H^1(\Sigma)^d}
\newcommand\HhalfB{\frac12,\Sigma}
\newcommand\Ldo{L^2_0(\Omega)}

\newcommand\dt{\tau}

\newcommand\triaB{\mathcal{S}_h}
\newcommand\Vh{\V_h}
\newcommand\Qh{Q_h}
\newcommand\Wh{\W_h}

\newcommand\Lh{{\boldsymbol \Lambda}_h}

 \newcommand*{\DOT}[1]{%
	\accentset{\mbox{\bfseries .}}{#1}} 
\renewcommand\c{\mathbf{c}}
\newcommand\f{\mathbf{f}}
\renewcommand\d{\mathbf{d}}
\newcommand\dd{\DOT\d}
\newcommand\n{\mathbf{n}}

\renewcommand\u{\mathbf{u}}
\renewcommand\v{\mathbf{v}}
\newcommand\w{\mathbf{w}}

\newcommand\z{\mathbf{z}}
\newcommand\X{\mathbf{X}}
\newcommand\V{\mathbf{V}}
\newcommand\W{\mathbf{W}}
\newcommand\ssigma{\boldsymbol{\sigma}}
\newcommand\llambda{\boldsymbol{\lambda}}
\newcommand\LL{\boldsymbol{\Lambda}}
\newcommand\mmu{{\boldsymbol{\mu}}}
\newcommand\pphi{{\boldsymbol{\phi}}}
\newcommand\eeps{{\boldsymbol{\epsilon}}}
\definecolor{BLUE}{rgb}{0,0,1}
\newcommand\St{\Sigma(t)}
\newcommand\Ot{\Omega(t)}
\newcommand\ns{\n}

\newcommand{\jump}[1]{\llbracket #1 \rrbracket }
\newcommand\rf{\rho^{\rm f}}
\newcommand\rs{\rho^{\rm s}\varepsilon}
\newcommand\Z{\boldsymbol{\mathcal{Z}}}
\newcommand\pt{\partial_t}
\newcommand\pd{\partial_\tau}
\newcommand\pdd{\partial_{\tau\tau}}
\renewcommand\L{\mathbf{L}}
\newcommand\T{\mathcal{T}}
\renewcommand\S{\mathcal{S}}
\newcommand\af{a^{\rm f}}
\newcommand\as{a^{\rm s}}
\newcommand\sh{s_h}
\renewcommand\Lsh{\L_h}
\newcommand\uh{\u_h}
\newcommand\ph{p_h}
\renewcommand\dh{\d_h}
\newcommand\ddh{\dd_h}
\newcommand\lh{\llambda_h}
\newcommand\pph{\pphi_h}
\newcommand\un{\u^n}
\newcommand\pn{p^n}
\newcommand\dn{\d^n}
\newcommand\ddn{\dd^n}
\renewcommand\ln{\llambda^n}
\newcommand\unp{\u^n_\Pi}
\newcommand\pnp{p^n_\Pi}
\newcommand\dnp{\d^n_\Pi}
\newcommand\ddnp{\dd^n_\Pi}
\newcommand\ddnup{\dd^{n-1}_\Pi}
\newcommand\lnp{\llambda^n_\Pi}
\newcommand\uhn{\u_h^n}
\newcommand\phn{p_h^n}
\newcommand\dhn{\d_h^n}
\newcommand\ddhn{\dd_h^n}
\newcommand\lhn{\llambda_h^n}
\newcommand\pphn{\pphi_h^n}
\newcommand\ddhnn{\dd_h^{n-\frac12}}
\newcommand\dnstar{\d^{n*}}
\newcommand\dhnstar{\d_h^{n*}}
\newcommand\xinstar{\boldsymbol{\xi}_h^{n*}}
\newcommand\uhnu{\u_h^{n-1}}

\newcommand\dhnu{\d_h^{n-1}}
\newcommand\ddhnu{\dd_h^{n-1}}

\newcommand\pphnu{\pphi_h^{n-1}}
\newcommand\tnp{\boldsymbol{\theta}^n_\Pi}
\newcommand\tnh{\boldsymbol{\theta}^n_h}
\newcommand\phinp{\phi^n_\Pi}
\newcommand\phinh{\phi^n_h}
\newcommand\xinp{\boldsymbol{\xi}^n_\Pi}
\newcommand\xinh{\boldsymbol{\xi}^n_h}
\newcommand\dxinp{\DOT{\boldsymbol{\xi}}^n_\Pi}
\newcommand\dxinh{\DOT{\boldsymbol{\xi}}^n_h}
\newcommand\omnp{\boldsymbol{\omega}^n_\Pi}
\newcommand\omnh{\boldsymbol{\omega}^n_h}
\newcommand\cchinh{\boldsymbol{\chi}_h^n}
\newcommand\tnuh{\boldsymbol{\theta}^{n-1}_h}
\newcommand\xinuh{\boldsymbol{\xi}^{n-1}_h}
\newcommand\dxinuh{\DOT{\boldsymbol{\xi}}^{n-1}_h}
\newcommand\Ehn{{E}^n_h}
\newcommand\Eho{{E}^0_h}
\newcommand\Dhn{{D}^n_h}

\newcommand\OO{{0,\Omega}}
\newcommand\OS{{0,\Sigma}}
\newcommand\PPi{\boldsymbol{\Pi}}
\newcommand\PV{\PPi_V}
\newcommand\PQ{\PPi_Q}
\newcommand\PW{\PPi_W}
\newcommand\PL{\PPi_\Lambda}
\theoremstyle{plain}
\newtheorem{lem}{Lemma}
\newtheorem{problm}{Problem}
\newtheorem{algo}{Algorithm}
\newtheorem{thm}{Theorem}

\theoremstyle{definition}
\theoremstyle{remark}
\newtheorem*{remark}{Remark}
\usepackage{subfigure}
\usepackage{colortbl}
\usepackage{pgfplots}
\usepackage{pgfplotstable}
\usepackage{tikz}
\usepackage{tkz-euclide}
\usetkzobj{all}
\usetikzlibrary{decorations.pathreplacing,shapes.misc,calc,intersections,arrows,external}
\usepgfplotslibrary{external}
\tikzsetexternalprefix{./figures/}
\usepackage{booktabs}

\graphicspath{{./figures/}}

\begin{document}

\title[Splitting schemes for a Lagrange multiplier formulation]
{Splitting schemes for a Lagrange multiplier formulation of FSI with
immersed thin-walled structure: stability and convergence analysis}

\author{Michele Annese}
\address{DICATAM, Università\`a degli Studi Brescia, 25123 Brescia, Italy}
\email{m.annese@unibs.it}
\author{Miguel A. Fern\'andez}
\address{Inria, 75012 Paris, France and
Sorbonne Universit\'e and CNRS,  LJLL UMR 7598, 75005 Paris, France}
\email{miguel.fernandez@inria.fr}
\author{Lucia Gastaldi}
\address{DICATAM, Universit\`a degli Studi di Brescia, 25123 Brescia, Italy}
\email{lucia.gastaldi@unibs.it}
\urladdr{http://lucia-gastaldi.unibs.it}

\begin{abstract}
% Body of abstract:
{ The numerical approximation of incompressible fluid-structure
interaction problems with Lagrange multiplier is generally based on strongly
coupled schemes. This delivers unconditional stability but at the expense of
solving a computationally demanding coupled system at each time-step. For the
case of the coupling with immersed thin-walled solids, we introduce a class of
semi-implicit coupling schemes which avoids strongly coupling without
compromising stability and accuracy. A priori energy and error estimates are
derived. The theoretical results are illustrated through  numerical experiments
in an academic benchmark.}
% Keywords:
{fluid-structure interaction, immersed boundary method,  
Lagrange multiplier, finite elements, time-splitting schemes.}
\end{abstract}
\maketitle

%%%%%%%%%%%%%%%%%%%%%%%%%%%
\section{Introduction}
%%%%%%%%%%%%%%%%%%%%%%%%%%%
\label{se:intro}

The numerical simulation of multi-physics systems coupling 
an incompressible viscous fluid with an immersed thin-walled elasitc solid is of major importance  in 
many engineering and living systems. Among the examples, we can mention the 
aeroelasticity of parachutes and sailing boats and the mechanics of capsules, 
biological cells and heart valves  (see, e.g., \cite{lombardi-et-al-2012,takizawa-tezduyar-12,eswaran-et-al-09,paidoussis-et-al-11,pozrikidis-10,heil-hazel-11,mullert-et-al-10,tian-et-al-14}).

 These coupled problems often feature large interface displacements, with potential contact between solids, so that  the 
 favored spatial discretization is mainly  based on  unfitted mesh approximations (the fluid mesh is not fitted to the fluid-solid interface). 
 Among these methods, the most popular are the  immersed boundary method (see,
e.g., \cite{peskin-02,newren-et-al-07,boffi-cavallini-gastaldi-11}) and the
fictitious domain method   (see, e.g.,
\cite{glowinski-pan-hesla-joseph-99,baaijens-01,dehart-et-al-03,astorino-et-al-09,kamensky-et-al-15,boffi2015finite,alauzet-et-al-15,BG2017,casquero-et-al-18}), which treat the solid in its natural Lagrangian formalism.    
We refer to \cite{boilevin-fernandez-gerbeau-19} for a recent numerical study which compares some of these approaches.

Over the last decade, significant advances have been achieved in the development and the analysis of time splitting schemes that avoid strong coupling without compromising stability and accuracy. The majority of these studies is limited to fitted fluid and solid meshes (see, e.g., \cite{fernandez-gerbeau-grandmont-07,quaini-quarteroni-07,badia-quaini-quarteroni-08,burman-fernandez-09,
 guidoboni-et-al-09,bukac-et-al-2012,lukacova-et-al-12,fernandez-13,fernandez-mullaert-vidrascu-13,banks-et-al-14b,fernandez-mullaert-vidrascu-15,fernandez-landajuela-vidrascu-15,landajuela-et-al-15}). Within the unfitted mesh framework, splitting schemes which efficiently avoid strong coupling are much more rare. 
The original time-stepping scheme of the immersed boundary method uncouples the 
fluid and solid time-marchings (actually, the solid solver is never called) but at the price of enforcing  severe time-step restrictions for stability 
(see, e.g., \cite{boffi-cavallini-gastaldi-11}).  The splitting schemes
reported in
\cite{burman-fernandez-14,alauzet-et-al-15,kim-lee-choi-18,kapada-et-al-18} are
also known to enforce severe time-step restrictions for stability/accuracy 
or to be sensitive to the amount of added-mass effect.

In the present paper, we introduce a semi-implicit coupling scheme
for a formulation based on the introduction of a Lagrange multiplier which
avoids the above mentioned issues. The proposed approach generalizes the ideas
introduced in \cite{Fernandez2013,fernandez-landajuela-20} to the case of
unfitted mesh approximations with Lagrange multipliers (see
\cite{BCG15,BG2017}). The analysis shows, in particular, that the scheme with
first-order extrapolation   yields unconditional stability and optimal
(first-order) accuracy in time.  To the best of our knowledge, this is the
first time that the full numerical analysis is addressed for linear
incompressible fluid-structure interaction problems with Lagrange multipliers.
Numerical experiments in an academic test case illustrate the behavior of the proposed approach.

%\cite{fernandez-landajuela-20}
%\cite{boilevin-fernandez-gerbeau-19}
%\cite{annese-17}
% 
%We will consider the finite element version of the immersed 
%boundary 

 The rest of the paper is organized as follows. 
Section~\ref{se:pb} presents the coupled system and its weak formulation with Lagrange multipliers.
The numerical methods are described in Section~\ref{sect:nummed}.
Section \ref{sec:analysis} presents the stability and the error analysis. 
Numerical evidence of the theoretical findings  is provided in Section~\ref{sect:exp}.
%Finally, a summary of the main conclusions is given in Section~\ref{sect:conclusions}.

%%%%%%%%%%%%%%%%%%%%%%%%%%%
\section{Problem setting and weak formulation}
%%%%%%%%%%%%%%%%%%%%%%%%%%%
\label{se:pb}
We consider fluid-structure interaction problems characterized by a thin-walled
structure immersed in an incompressible viscous fluid. Let $\Omega\subset\RE^d$, $d=2,3$,
be a fixed bounded domain with Lipschitz continuous boundary $\Gamma$. 
In order to describe the dynamics of the structure immersed in the fluid, we use
a Lagrangian framework. The elastic thin-walled structure is represented by its {mid-surface (i.e., a curve if $d=2$ or a surface if $d=3$)}. Let $\Sigma\subset\RE^{d}$ be the reference
configuration of the thin-walled structure {mid-surface}. Its current position, denoted by $\St$,
is obtained as the image of the deformation mapping
$\pphi(\cdot,t):\Sigma\to\St\subset\Omega$.
The domain occupied by the fluid is denoted $\Ot = \Omega \backslash \Sigma(t)$ and its boundary by 
$\partial\Ot=\Gamma\cup\St$.
We assume that the interface $\St$ is oriented by a unitary normal vector field  $\ns$,  
which induces a positive and a negative side in the fluid domain $\Omega(t)$, with respective 
unit normals  $\ns^{\rm +} := \ns $ and $\ns^{\rm -} := -\ns  $ on $\Sigma(t)$. 
Thus, 
we can define the positive and negative sided-restrictions to $\Sigma(t)$ of  a given  field $f$ 
defined in $\Omega(t)$,  as 
$f^{+}({\bf x}) := \lim_{\xi \to 0^{+}} f({\bf x} + \xi {\bf n}^{+}), \quad
f^{-}({\bf x}) := \lim_{\xi \to 0^{-}} f({\bf x} + \xi {\bf n}^{-} )$,  $
\forall{\bf x} \in \Sigma(t) $, and the normal jump 
%We shall also make use of the following jump operators across the interface $\Sigma(t)$:
$%\jump{f} := f^{+} - f^{-}, \quad
 \jump{f{\bf n}  } := f^{+}{\bf n}^{\rm{+}} + f^{-}{\bf n}^{-}. $ 

\begin{figure}[!h]

  \centering
\includegraphics{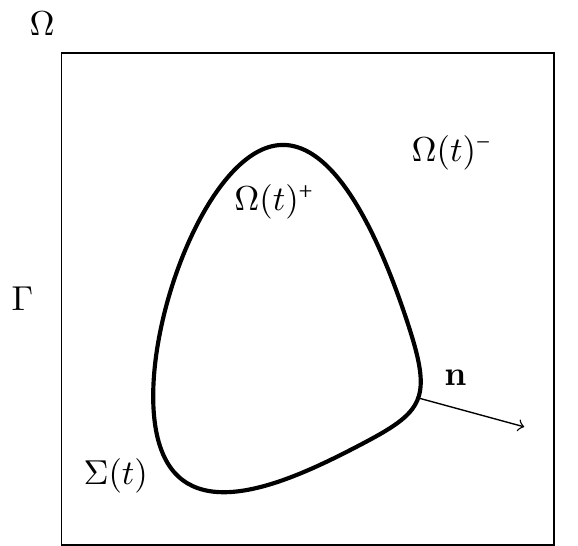}
  \caption{Geometrical configuration of the FSI problem.}
    \label{fig:domain}
\end{figure}

We adopt the Eulerian framework to write the Navier-Stokes equations which
govern the dynamics of the incompressible fluid and the Lagrangian framework
for the elastodynamics of the thin structure. We denote by
$\ssigma^{\rm f}:=-p\mathbb{I}+2\mu\eeps(\u)$ the Cauchy stress tensor for
incompressible fluids, where $\eeps(\u):=(\Grad\u+\Grad\u^\top)/2$ is
the strain tensor. We assume that the abstract linear surface differential operator
$\L$ describes the solid elastic effects.
Hence, we have the following
coupled problem. 
\begin{problm}
\label{pb:coupled}
Given $\u_0$, $\d_0$ and $\d_1$,
for $t\in]0,T]$, find the fluid velocity $\u$, the pressure $p$ in $\Ot$,
the solid displacement $\d$ and velocity $\dd$ in $\Sigma$ such that:
\begin{itemize}
\item\textbf{Fluid sub-problem}:
\begin{equation}
\label{eq:NS}
\left\{
\aligned
&\rf\big(\pt\u+\u\Grad\u\big)-\div\ssigma^{\rm f}=\bf 0
&& \text{in }\Ot,\\
&\div\u=0 && \text{in }\Ot,\\
%&\jump{\u}=0 && \text{on }\St\times]0,T]\\
%&\jump{\ssigma_f\n}=\f && \text{on }\St\times]0,T]\\
&\u=\bf 0&& \text{on }\Gamma.
\endaligned\right.
\end{equation}
\item\textbf{Solid sub-problem}:
\begin{equation}
\label{eq:solid}\left\{ 
\aligned
&\rs\pt\dd+\L\d=\f_{\Sigma}&&
\text{in }\Sigma,\\
&\pt\d=\dd && \text{in }\Sigma,\\
&\d=\bf 0&&\text{on }\partial\Sigma.
%&\P\mathbf{N}=0&&\text{on }\partial\Sigma_t\times]0,T]\\
\endaligned\right. 
\end{equation}
\item\textbf{Interface coupling conditions}:
\begin{equation}
\label{eq:coupling}\left\{ 
\aligned
&\pphi= % \pphi_0 
\boldsymbol  I
+\d\quad\text{in}\quad\Sigma,\\
&\u\circ\pphi=\dd\quad \text{on}\quad\Sigma,\\
&\int_\Sigma \f_{\Sigma} \cdot \w=-\int_{\St}{\jump{\ssigma^{\rm f}\n}} \cdot \w\circ\pphi_t^{-1}\quad\forall 
\w:\Sigma \rightarrow \RE^d \text{ smooth.}
\endaligned\right. 
\end{equation}
\item\textbf{Initial conditions}:
\begin{equation}
\label{eq:init}\left\{ 
\aligned
&\u(\cdot,0)=\u_0&&\text{ in }\Omega(0),\\
&\d(\cdot,0)=\d_0,\quad \dd(\cdot,0)=\d_1&&\text{ in }\Sigma.
\endaligned\right. 
\end{equation}
\end{itemize}
%$\pphi_0=\pphi(\cdot,0)$ and $\Sigma(0)=\pphi_0(\Sigma)$.
%
\end{problm}
%In Problem~\ref{pb:coupled}, the notation 

The relation \eqref{eq:coupling}$_2$ enforces the so-called kinematic coupling condition 
(continuity of velocity across the interface), while \eqref{eq:coupling}$_3$  states 
that the tractions along the immersed interface have to be equilibrated (dynamic coupling).
\begin{remark}
In Problem~\ref{pb:coupled}, the solid mid-surface is fully identified 
with the fluid-solid interface $\Sigma$, by neglecting the solid thickness effects in the interface coupling.
This is a rather widespread modeling assumption when coupling thin-walled solids with a 3D media (see, e.g., \cite{chapelle-ferent-03,landajuela-et-al-17}). Yet, in the context of immersed boundary methods, a correction term 
is often introduced to remove the \emph{across-the-thickness} additional fluid mass (see, e.g., \cite{BCG2011,BCG15}). 
\end{remark}

In the following, we introduce the weak formulation of {Problem~\ref{pb:coupled}}. We {shall} make use
of the standard Sobolev space $H^1_0(D)^d$ of the vector valued functions in
$H^1(D)$ which vanish on the boundary $\Gamma$, and of $L^2_0(D)$ the subspace
of functions in $L^2(D)$ with zero mean value in $D$. The corresponding norms
are denoted by $\|\cdot\|_1$ and $\|\cdot\|_0$, respectively.  The scalar
product in $L^2(D)$ is denoted by $(\cdot,\cdot)_D$. The subscript is dropped
if $D=\Omega$. We denote by $\W\subseteq\Hub$ the subspace of admissible
deformation which satisfy the Dirichlet boundary condition for the solid.
Moreover, we shall use the bilinear forms $a^{\rm f}:(\Huo\times  \Ldo) \times(\Huo\times  \Ldo) \to\RE$ and
$a^{\rm s}:\W\times\W\to\RE$ and the   trilinear form
$b:\Huo\times\Huo\times\Huo\to\RE$, defined by:
\[
\aligned
&\af\big((\u,p),(\v,q)\big):=2\mu \big(\eeps(\u),\eeps(\v)\big)  -(\div\v, p) +  (\div\u, q),\\\
&b(\z,\u,\v):=\frac{\rf}2\big((\z\cdot\grad\u,\v)-(\z\cdot\grad\v,\u)\big),\quad \as(\d,\w):=(\L\d,\w)_\Sigma=(\d,\L\w)_\Sigma.
\endaligned
\]
We assume that $\as$ is symmetric, continuous and coercive on $\W$ with
associated norm $\|\cdot\|_{\rm s}^2=\as(\cdot,\cdot)$ and that it commutes with the
time derivative, that is $\pt\as(\w(t),\w(t))=2\as(\pt\w(t),\w(t))$.

By multiplying  the relations~\eqref{eq:NS} and~\eqref{eq:solid} by 
${\bf v} \in \Huo$, $q\in \Ldo $, $\w \in \W$ and after integration by parts, and, taking into account the boundary
conditions, we obtain:
\begin{equation}\label{eq:nobc}
\aligned
&\rf\left(\pt\u,\v\right)+{b(\u,\u,\v)}+ 
\af\big((\u,p),(\v,q)\big)-({\jump{\ssigma_f\n}},\v)_{\St}=0,\\
&\rs(\pt\dd,\w)_\Sigma+\as(\d,\w)=(\f_{\Sigma},\w)_\Sigma.
\endaligned
\end{equation}
These relations are coupled through conditions~\eqref{eq:coupling}. 
In this work, we will enforce {the kinematic coupling \eqref{eq:coupling}$_1$} variationally using Lagrange multipliers (see, e.g., \cite{BCG15}).  To this purpose 
we introduce the Lagrange multiplier space $\LL:=(\HhalfB)'$, the trace space $\Z:=\HhalfB$ and the bilinear form $
 c:\LL\times\Z\to\RE$, defined as 
 \begin{equation}
\label{eq:defc}
c(\mmu,\w):=\langle\mmu,\w\rangle\quad 
\forall\mmu\in\LL, \w\in\Z,
\end{equation}
{where $\langle\cdot,\cdot\rangle$ denotes the duality pairing between $\LL$ and $\Z$.}
%
%Moreover, we introduce a Lagrange multiplier $\llambda\in\LL$, so that 
Problem~\ref{pb:coupled} can hence be formulated in weak form as follows:
\begin{problm}
\label{pb:weak}
Given $\u_0$, $\d_0$, $\d_1$  with $\u_0\circ\pphi_0=\d_1$ in
$\Sigma$, for $t\in]0,T]$, find $\u(t)\in\Huo$, $p(t)\in\Ldo$, $\d(t)\in{\W}$,
$\dd(t)\in{\W}$ and $\llambda(t)\in\LL$ such that:
\begin{equation}
\label{eq:weak}
\left\{
\aligned
&\rf(\pt\u(t),\v)+b(\u(t),\u(t),\v)+\af\big((\u(t),p(t)),(\v,q)\big)\\
&\quad+c(\llambda(t),(\v\circ\pphi)(t))-c(\mmu,(\u\circ\pphi)(t)-\dd(t))=0\\
& \hspace{4cm}\forall(\v,q,\mmu)\in\Huo\times \Ldo\times \LL,\\
&\rs(\pt\dd(t),\w)_\Sigma+\as(\d(t),\w)=c(\llambda(t),\w) 
\qquad\qquad\forall\w\in{\W},\\
&\dd(t)= \pt\d(t),\\
&\pphi(t)=\boldsymbol I+\d(t)\text{ in }\Sigma,\\
&\u(0)=\u_0\text{ in }\Omega,\ 
\d(0)=\d_0\text{ in }\Sigma,\ 
\dd(0)=\d_1\text{ in }\Sigma.
\endaligned\right.
\end{equation}
\end{problm}
Comparing Problem~\ref{pb:weak} with the integral formulation \eqref{eq:nobc}, we
see that the Lagrange multiplier corresponds to the fluid-structure interaction
forces acting on the structure $\llambda=\f_{\Sigma}$. By taking $\v=\u(t)$, $q=p(t)$, $\w=\dd(t)$ and $\mmu=\llambda(t)$ in
Problem~\ref{pb:weak}, we have the following energy estimate:
\begin{equation}
\label{eq:energy}
\aligned
&\rf\|\u(t)\|_\OO^2+4\mu\int_0^t\|\eeps(\u(s))\|^2_\OO {\rm d}s
+\rs\|\dd(t)\|^2_\OS+\|\d(t)\|_{\rm s}^2\\
&\qquad=\rf\|\u_0\|^2_\OO+\|\d_0\|_{\rm s}^2+\rs\|\dd_1\|^2_\OS.
\endaligned
\end{equation}

%%%%%%%%%%%%%%%%
\section{Numerical methods}\label{sect:nummed}
%%%%%%%%%%%%%%%%

This sections is devoted to the numerical approximation of Problem~\ref{pb:weak}. The  next 
section presents the spatial discretization, using the immersed boundary/fictitious domain  finite element method 
with Lagrange multipliers (see \cite{BCG15,BG2017}). Section~\ref{se:splitting} is devoted 
to the time discretization. In particular,   new splitting schemes are introduced 
by generalizing the ideas introduced in \cite{Fernandez2013,fernandez-landajuela-20}.

\subsection{Unfitted mesh based semi-discretization in space}
The weak treatment of the interface coupling in Problem~\ref{pb:weak}, through the Lagrange multiplier, 
enables the arbitrary choice of the fluid and solid meshes. The main difficulty lies in the computation of the  
the coupling terms $c(\mmu,\v\circ\pphi)$, which require to evaluate the velocity basis functions
composed with the mapping $\pphi$, and, consequently, to intersect the current
configuration of the immersed solid with the fixed underlying fluid mesh. 
{This problem also arises in alternative unffited mesh methods  (see, e.g., \cite{AFFL2016}).}

Let us introduce the finite element spaces we shall use to
discretize the problem.
We can choose either a pair of space $\V_h$ and $Q_h$ which satisfy the inf-sup
conditions for the Stokes equations or stabilized finite elements.
In this paper, we choose the $\mathbb{P}_1/\mathbb{P}_1$ stabilized elements defined as follows. 
Let $\T_h$ be a regular subdivision of $\Omega$ into
triangles if $d=2$, tetrahedrons if $d=3$ and let $\S_h$ be a
regular subdivision of the reference domain $\Sigma$ into segments if $d=2$ or
triangles if $d=3$. We denote by $h_{\rm f}$ and $h_{\rm s}$ the meshsizes of $\T_h$ and
$\S_h$, respectively.  
We introduce the following finite element spaces
\begin{equation}
\label{eq:spaces}
\aligned
&\V_h=\{\v\in\Huo: \v|_K\in\mathbb{P}_1^d\ \forall K\in\T_h\},\\
&Q_h=\{q\in\Ldo: q|_K\in \mathbb{P}_1\ \forall K\in\T_h\},\\
&\W_h=\{\w\in\W: \w|_K\in \mathbb{P}_1^d\ \forall K\in\S_h\},\\
&\Lh=\{\mmu\in\Hub: \mmu|_K\in \mathbb{P}_1^d\ \forall K\in\S_h\},
\endaligned
\end{equation} 
where $ \mathbb{P}_1(K)$ stands for the space of affine polynomials on $K$. We use the 
stabilization technique introduced in \cite{brezzipitkaranta}, by adding the following term
in the discrete counterpart of~\eqref{eq:weak}:
\begin{equation}
\label{eq:stabilization}
\sh(p,q)=\gamma\sum_{K\in\T_h}h_K^2(\nabla p,\nabla q)_K\quad\forall p,q\in Q_h,
\end{equation} 
where $\gamma>0$ is a suitable {user-defined} constant. We shall use also
the broken norm $|q|_{\sh}^2=\sh(q,q)$, for all $q\in Q_h$, and the discrete Stokes bilinear form 
$$
\af_h\big((\u,p),(\v,q)\big) := \af\big((\u,p),(\v,q)\big) + \sh(p,q).
$$
The space semi discrete approximation of Problem~\ref{pb:weak} then reads:
\begin{problm}
\label{pb:weakh}
Given $\u_{0,h}\in\V_h$, $\d_{0,h}\in\W_h$, $\d_{1,h}\in\W_h$ and
$\pphi_{0,h}\in\W_h$, for $t\in[0,T]$, find $\uh(t)\in\V_h$, $\ph(t)\in Q_h$,
$\dh(t)\in\W_h$, $\ddh(t)\in\W_h$ and $\lh(t)\in\Lh$ such that:
\begin{equation}
\label{eq:weakh}
\left\{ 
\aligned
&\rf(\pt\uh(t),\v)+b(\uh(t),\uh(t),\v)+\af_h\big((\uh(t),p_h(t)),(\v,q)\big)\\
&\quad+c(\lh(t),\v\circ\pph(t)) - c(\mmu,(\uh\circ\pph)(t)-\ddh(t)) =0\\
&\hspace{4.3cm}\qquad\forall(\v,q,\mmu)\in\V_h\times Q_h\times\Lh,\\
&\rs(\pt\ddh(t),\w)_\Sigma+\as(\dh(t),\w)=c(\lh(t),\w)
\qquad\forall\w\in\W_h,\\
&\ddh(t) = \pt\dh(t),\\
&\pph(t)=\boldsymbol I+\dh(t)\text{ in }\Sigma,\\
&\uh(0)=\u_{0,h}\text{ in }\Omega,\ 
\dh(0)=\d_{0,h}\text{ in }\Sigma,\ 
\ddh(0)=\d_{1,h}\text{ in }\Sigma.
\endaligned\right.
\end{equation}
\end{problm}
Using the same argument as in the continuous case, we easily obtain the
discrete energy estimate:
\begin{equation}
\label{eq:energyh}
\aligned
&\rf\|\uh(t)\|_\OO^2+4\mu\int_0^t\|\eeps(\uh(s))\|^2_\OO{\mathrm d}s
+\rs\|\ddh(t)\|^2_\OS+\|\dh(t)\|_{\rm s}^2\\
&\qquad+\int_0^t|\ph(s)|^2_{\sh}{\mathrm d}s
=\rf\|\u_{0,h}\|^2_\OO+\|\d_{0,h}\|_{\rm s}^2+\rs\|\dd_{1,h}\|^2_\OS.
\endaligned
\end{equation}
\subsection{Time discretization and splitting schemes}
\label{se:splitting}
In this subsection, we present a semi-implicit time discretization of
Problem~\ref{pb:weakh}. Given a positive integer $N$, let $\dt=T/N$ be the time
step, and $t_n=n\dt$ for $n=0,\dots,N$.
For a given function $g$ depending on $t$, we adopt the following notation
\[
g^n:=g(t_n), \qquad \pd g^n:=\frac{g^n-g^{n-1}}{\tau},
\qquad
\pdd g^n:
=\frac{g^n-2g^{n-1}+g^{n-2}}{\tau^2}.
\]

\subsubsection{Strongly coupled scheme}
Using the backward Euler scheme and evaluating the term along the structure location at
the previous time step, we have the following strongly coupled scheme (see \cite{BCG15}).
\begin{algo}
\label{pb:fullydiscr}
Let  $\u_{0,h}\in\V_h$, $\d_{0,h}\in\W_h$, $\d_{1,h}\in\W_h$ and
$\pphi_{0,h}\in\W_h$ be given.  We set $\uh^0=\u_{0,h}$,  $\dh^0=\d_{0,h}$ and $\ddh^0=\d_{1,h}$.
For $n=1,\dots,N$, perform the following steps:
\begin{description}
\item Step 1. Find $(\uhn,\phn, \lhn, \dhn,\ddhn) \in \V_h \times Q_h\times \Lh\times \W_h\times \W_h$ such that:
\begin{equation}
\label{eq:fullydiscr}\left\{ 
\aligned
&\rf(\pd\uhn,\v)+b(\uhnu,\uhn,\v)+\af_h\big((\uhn,\phn),(\v,q)\big)\\
&\quad+ c(\lhn,\v\circ\pphnu)
-c(\mmu,\uhn\circ\pphnu-\ddhn)=0\\
&\hspace{4cm}\forall(\v,q,\mmu)\in\V_h\times Q_h \times \Lh,\\
&\rs(\pd\ddhn,\w)_\Sigma+\as(\dhn,\w)= c(\lhn,\w)
\qquad\forall\w\in\W_h,\\
&\ddhn = \pd\dhn.
\endaligned\right.
\end{equation}
\item Step 2. Update interface:
$
\pphn=\boldsymbol I+\dhn.
$
\end{description} 
\end{algo}
At each time step, Algorithm \ref{pb:fullydiscr} involves the solution of the monolithic system \eqref{eq:fullydiscr} with a saddle point structure. The existence and uniqueness of the continuous
and discrete versions of such problem have been analyzed in~\cite{BG2017} in
the case of inf-sup stable finite element discretization of the Stokes equation, and
optimal a priori error estimates have been deduced according to the theory of
discretization of saddle point problems (see, e.g.~\cite{BBF}). More recently,
the above analysis has been extended to cover the case of stabilized $\mathbb{P}_1/\mathbb{P}_1$ 
elements in~\cite{annese}.

Testing \eqref{eq:fullydiscr} with $(\v,q,\mmu,\w)=(\uhn,\phn,\lhn,\ddhn)$ and using the velocity-dis\-place\-ment relation 
$\ddhn= \pd\dhn$, the energy estimate \eqref{eq:energyh} extends also to Problem~\ref{pb:fullydiscr} in the following form (see also \cite[Proposition 4]{BCG15}):
\begin{multline*}
\rf\|\uhn\|^2_\OO+4\mu\sum_{m=1}^N\dt\|\eeps(\uh^m)\|^2_\OO 
+\rs\|\ddhn\|^2_\OS+ \|\dhn\|^2_{\rm s}+2\sum_{m=1}^N\dt|\ph^m|^2_{\sh}\\
\le \rf \|\u_h^0\|^2_\OO + \rs\|\ddh^0\|^2_\OS.
\end{multline*}
This guarantees the unconditional stability of the strongly coupled scheme provided by Algorithm~\ref{pb:fullydiscr}. 
It should be noted that this superior stability comes at the cost of solving  at each time step a high-dimensional heterogenous system, 
which can be ill conditioned and computational demanding.

%The following sections will be devoted to the introduction of splitting schemes 
%in the spirit of~\cite{Fernandez2013,FL2015,FLV2015,FM2016,AFFL2016}.
%%

\subsubsection{Splitting schemes}
%Despite the monolithic algorithm presented in Problem~\ref{pb:fullydiscr}
%enjoys unconditional stability, the numerical implementation requires the
%writing of specific codes and does not allow for the use of existing and
%well-tested solver for fluid and solid problems. Therefore, in the following,
%we propose splitting schemes able to avoid serious degradation of stability and
%accuracy.
In order to circumvent the computational complexity of the strong coupling (Algorithm~\ref{pb:fullydiscr}), 
the time discretizations of the original immersed boundary method introduced a significant time splitting in the computation 
of the fluid and solid fields  (see, e.g., \cite{Peskin1977,TP1992}).  
Basically, the idea consisted in treating explicitly the solid 
elastic contributions within the fluid and then retrieving the solid displacement directly from the interpolation of the fluid velocity into the solid grid.  
The fundamental drawback 
of this approach is that restrictive CFL-like conditions are required for
stability (see, e.g., \cite{TP1992,SW1999,BGH2007,BGH2007a,BCG2011}).
Within the context of the spatial approximation provided by Problem~\ref{pb:weakh}, this solution procedure would take the following form:
\begin{algo}
\label{pb:explicit}
Let  $\u_{0,h}\in\V_h$, $\d_{0,h}\in\W_h$, $\d_{1,h}\in\W_h$ and
$\pphi_{0,h}\in\W_h$ be given.  We set $\uh^0=\u_{0,h}$,  $\dh^0=\d_{0,h}$ and $\ddh^0=\d_{1,h}$
For $n=1,\dots,N$, perform the following steps:
\begin{description}
\item Step 1. Find $(\uhn,\phn,\lhn,\ddhn)\in\V_h\times Q_h \times \Lh\times \W_h$ such that:
\begin{equation}
\label{eq:split1expli}\left\{ 
\aligned
&\rf(\pd\uhn,\v)+b(\uhnu,\uhn,\v)+\af_h\big((\uhn,\phn),(\v,q)\big)\\
&\quad+c(\lhn,\v\circ\pphnu)-c(\mmu,\uhn\circ\pphnu-\ddhn)=0&&\quad\forall(\v,q,\mmu)\in\V_h\times Q_h\times \Lh,\\
%&(\div\uhn,q)+\sh(\phn,q)=0&&\quad\forall q\in Q_h,\\
&\rs(\pd\ddhn ,\w)_\Sigma=c(\lhn,\w)-\as(\dh^{n-1},\w)
&&\quad\forall\w\in\W_h.
%&\c(\mmu,\uhn\circ\pphnu-\ddhnn)=0&&\quad\forall\mmu\in\Lh.
\endaligned\right. 
\end{equation}
\item Step 2.  Update solid displacement: $\dh^n = \dh^{n-1}+\tau \ddhn$. 

\item Step 3. Update interface:
$
\pphn=\boldsymbol I+\dhn.
$
\end{description} 
\end{algo}
%Note that, in the strongly coupled scheme \eqref{eq:fullydiscr}, the solid force acting on the fluid (i.e., the Lagrange multiplier) can be
%decomposed  into inertial forces, due to the acceleration of the solid mass, and 
%elastic forces, due to the solid deformation. 
The main idea behind the splitting of Algorithm~\ref{pb:explicit}  is to treat separately the two forcing terms:
the solid inertial and elastic contributions are, respectively, implicitly and explicitly coupled with the fluid. 
The first avoids added-mass stability issues while the second introduces a certain degree of splitting in the time-discretization. 
Note that, contrarily to Algorithm~\ref{pb:fullydiscr}, the solid solver is never called in Algorithm~\ref{pb:explicit}.  In fact, this 
is the source of instability in the scheme. Indeed, a simple argument shows that by testing \eqref{eq:split1expli} with 
$(\v,q,\mmu) =(\uhn,\phn,\lhn)$ we get the  energy estimate
\begin{multline}\label{eq:estimacfl}
\rf\|\uhn\|^2_\OO+4\mu\sum_{m=1}^N\dt\|\eeps(\uh^m)\|^2_\OO 
+\rs\|\ddhn\|^2_\OS+ \|\dhn\|^2_{\rm s}+2\sum_{m=1}^N\dt|\ph^m|^2_{\sh}\\
\le \rf \|\u_h^0\|^2_\OO + \rs\|\ddh^0\|^2_\OS
+ \|\dh^0\|^2_{\rm s}+ \|\dh^0\|^2_{\rm s} + 2\sum_{m=1}^N\tau \as(\dh^{m}-\dh^{m-1},\ddh^m) .
\end{multline}
Note that the last term is nothing but the artificial power generated by the explicit treatment of the 
solid elastic contributions in  \eqref{eq:split1expli}.
This can be controlled, but at the expense of enforcing restrictive CFL-like stability conditions (see Remark~\ref{re:cfl} for the details).

\begin{remark}\label{re:cfl}
In order to get stability from \eqref{eq:estimacfl}, the last term can be
controlled via a Gronwall type argument. Indeed, it suffices to  use the
continuity of $\af$ and a discrete inverse inequality to obtain
\[
\aligned
&\tau \as(\dh^{m}-\dh^{m-1},\ddh^m) 
= \tau^2 \|  \ddh^m \|_{\rm s}^2
\leq \tau^2 \beta^{\rm s} \|  \ddh^m \|_{1,\Sigma}^2\\
&\qquad\leq\tau\frac{\tau C_{\rm I}\beta^{\rm s} }{(h^{\rm s})^2}\|\ddh^m\|_{0,\Sigma}^2
=\tau\rs\frac{\tau C_{\rm I} \beta^{\rm s} }{\rs (h^{\rm s})^2}\ 
\|\ddh^m\|_{0,\Sigma}^2,
\endaligned
\]
where $\beta^{\rm s},C_{\rm I}>0$ respectively denote the continuity and inverse inequality constants. Hence, the energy stability follows by inserting this estimate into \eqref{eq:estimacfl} and 
by applying the discrete Gronwall lemma (see Lemma~\ref{le:Gronwall} below), under the parabolic CFL condition 
$$
\tau  \leq \frac{\rs}{C_{\rm I} \beta^{\rm s}} (h^{\rm s})^2. 
$$
\end{remark}

In this paper, we propose to avoid the stability issues of
Algorithm~\ref{pb:explicit} by generalizing the arguments of
\cite{Fernandez2013,FL2019} to the unfitted mesh approximation provided by
Problem~\ref{pb:weakh}.  Basically, the idea consists in replacing the
displa\-cement-velocity relation of Step 2 in Algorithm~\ref{pb:explicit} by a
full call of the solid solver, using the fluid load provided by Step 1. The
resulting solution procedure is detailed in Algorithm~\ref{pb:split} below,
where the symbol $\dhnstar$ denotes an extrapolation of the solid velocity,
namely,  
\begin{equation}
\label{eq:extrapolation}
\dhnstar = \left\{
\begin{array}{ll}
%{\bf 0} & \text{if } r=0,\\ 
\dh^{n-1} & \text{if } r=1,\\ 
\dh^{n-1}+\tau \ddh^{n-1} & \text{if } r=2.
\end{array}
\right.
\end{equation}
Note that, if $\dh$ is a smooth function of $t$, then the above
extrapolations provide an approximation error of order $r$.

\begin{algo}
\label{pb:split}
Let  $\u_{0,h}\in\V_h$, $\d_{0,h}\in\W_h$, $\d_{1,h}\in\W_h$ and
$\pphi_{0,h}\in\W_h$ be given.  We set $\uh^0=\u_{0,h}$,  $\dh^0=\d_{0,h}$ and $\ddh^0=\d_{1,h}$.
For $n=1,\dots,N$, perform the following steps:
\begin{description}
\item Step 1. Find $(\uhn,\phn,\lhn,\ddhnn)\in \V_h\times Q_h \times \Lh \times \Wh$ such that
\begin{equation}
\label{eq:split1}\left\{ 
\aligned
&\rf(\pd\uhn,\v)+b(\uhnu,\uhn,\v)+\af_h\big((\uhn,\phn),(\v,q)\big)\\
&\quad+c(\lhn,\v\circ\pphnu)-c(\mmu,\uhn\circ\pphnu-\ddhnn)=0\\
&\hspace{5cm}\quad\forall(\v,q,\mmu)\in\V_h\times Q_h\times \Lh,\\
&\frac{\rs}{\dt}(\ddhnn-\ddhnu,\w)_\Sigma=c(\lhn,\w)-\as(\dhnstar,\w)
\qquad\forall\w\in\W_h,\\
\endaligned\right. 
\end{equation}
\item Step 2. Find $(\dhn,\ddhn)\in\W_h\times \Wh$ such that
\begin{equation}
\label{eq:split2}\left\{ 
\aligned
%&\frac{\rs}{\dt}(\ddhn-\ddhnn,\w)_\Sigma+\as(\dhn-\dhnstar,\w)=0
&\rs(\pd\ddhn,\w)_\Sigma+\as(\dhn,\w)= c(\lhn,\w)
&&\quad\forall\w\in\W_h,\\
&\ddhn = \pd\dhn.
\endaligned\right. 
\end{equation}
\item Step 3. Update interface:
 $
\pphn=\boldsymbol I+\dhn.
$
\end{description}
\end{algo}

\begin{remark}
For a specific choice of the Lagrange multipliers space  $\Lh$, the unknowns 
$\lhn$ and $\ddhnn$ in step 1 of Algorithm~\ref{pb:split} 
can be eliminated in terms of $(\uhn,p^n_h)$ 
(see \cite{boilevin-fernandez-gerbeau-19a}).
\end{remark}

%%%%%%%%%%%%%%%%%%%
\section{Numerical analysis of Algorithm~\ref{pb:split}}\label{sec:analysis}
%%%%%%%%%%%%%%%%%%%

This section is devoted to the numerical analysis of the splitting schemes given by Algorithm~\ref{pb:split}. 
The energy stability properties  of the methods are analyzed in the next section, while Section~\ref{se:error} 
provides an \emph{a priori} error analysis in the case of a linearized version of Problem~\ref{pb:coupled}. 

\subsection{Energy stability analysis}
This  section is devoted to the analysis of the stability properties
of Algorithm~\ref{pb:split}. To this purpose, we first recall  some auxiliary results
which will be used later. The first one is a quite general version of
discrete Gronwall's Lemma from~\cite{HR1990}.
\begin{lem}
\label{le:Gronwall}
Let $\dt,B$ and $a_m,b_m,c_m,\gamma_m$, for integers $m\ge1$, be non negative
numbers such that, for $n\ge1$ 
\[
a_n+\dt\sum_{m=1}^{n}b_m\le
\dt\sum_{m=1}^{n}\gamma_m\,a_m +\dt\sum_{m=1}^{n}c_m+B.
\]
Suppose that $\dt\gamma_m<1$ for all $m\ge1$. Then, for $n\ge1$ it holds
\[
a_n+\dt\sum_{m=1}^{n}b_m\le
\mathrm{exp}\left(\dt\sum_{m=1}^{n}\frac{\gamma_m}{1-\tau\,\gamma_m}\right)
\left(\tau\sum_{m=1}^{n}c_m+B\right).
\]
\end{lem}
We define a discrete counterpart $\Lsh:\W\to\W_h$ of the elastic operator $\L$ as follows:
\begin{equation}
\label{eq:defLh}
(\Lsh\w,\z)_{\Sigma}=\as(\w,\z)
\end{equation} 
for all $\z\in\W_h.$
In~\cite[Lemma~1]{Fernandez2013} the following properties of $\Lsh$ have been
proved:
\begin{lem}
\label{le:inverse}
Let $\L\w\in L^2(\Sigma)^d$, then
\begin{equation}
\label{eq:stabLh}
\|\Lsh\w\|_{\OS}\le C\|\L\w\|_{\OS}.
\end{equation}
Under the assumption that the mesh $\S_h$ is quasi-uniform, then there exists a
positive constant $C_{\rm I}$ such that for all $\w_h\in\W_h$ it holds true
\begin{equation}
\label{eq:inverse}
\|\Lsh\w_h\|_{\rm s}\le  C_{\rm I}h_{\rm s}^{-2}\|\w_h\|_{\rm s}.
\end{equation}
\end{lem}
From equation~\eqref{eq:split2} we obtain the following characterization of the
intermediate value of the displacement velocity in terms of the solid velocity
and displacement:
\begin{lem}
\label{le:dn12}
Let $\{ (\uhn,\phn,\lhn,\ddhnn,\dhn,\ddhn)\}_{n\geq 1} \subset \V_h\times Q_h\times\Lh\times\W_h\times \W_h\times\W_h$ be given by 
Algorithm~\ref{pb:split}. We have 
\begin{equation}
\label{eq:dn12}
\ddhnn=\ddhn+\frac{\dt}{\rs}\Lsh(\dhn-\dhnstar).
\end{equation}
\end{lem}
\begin{proof} By subraction \eqref{eq:split2}$_1$ form \eqref{eq:split1}$_2$ we get 
\begin{equation}\label{eq:end-of-step-vel}
\frac{\rs}{\dt}(\ddhnn-\ddhn,\w)_\Sigma -\as(\dhn-\dhnstar,\w)= 0
\end{equation}
for all $\w\in \Wh$. The relation hence follows by from the definition of the discrete elastic
operator~\eqref{eq:defLh}.
\end{proof}

The energy estimate for Algorithm~\ref{pb:split} is given in terms of the
\emph{discrete energy} $\Ehn$ and of the \emph{discrete dissipation} $\Dhn$
defined, respectively, as
\begin{equation}
\label{eq:discren}
\aligned 
&\Eho=\rf\|\u_{0,h}\|^2_\OO+\rs\|\d_{1,h}\|^2_\OS+\|\d_{0,h}\|^2_{\rm s},\\
&\Ehn=\rf\|\uhn\|^2_\OO+\rs\|\ddhn\|^2_\OS+\|\dhn\|^2_{\rm s},\\
&\Dhn=\sum_{m=1}^n\dt\left(4\mu\|\eeps(\uh^m)\|^2_\OO+2|\ph^m|^2_{\sh}\right).
\endaligned
\end{equation}
The following theorem states that the splitting scheme is unconditionally
stable for $r=1$, while for
$r=2$ it is conditionally stable.
\begin{thm} 
\label{th:stab}
Let $\{ (\uhn,\phn,\lhn,\ddhnn,\dhn,\ddhn)\}_{n\geq 1} \subset \V_h\times Q_h\times\Lh\times\W_h\times \W_h\times\W_h$ be given by 
Algorithm~\ref{pb:split}.

\begin{itemize}
%\item[Scheme with $r=0$, $\dhnstar=0$]
%%
%\begin{equation}
%\label{eq:enr0}
%\Ehn+\Dhn\le \Eho;
%\end{equation}
%%
\item Scheme with $r=1$.  For $n\ge1$, we have 
\begin{equation}
\label{eq:enr1}
\Ehn+\Dhn+\dt^2\|\ddhn\|^2_{\rm s}+\frac{\dt}{2\rs}\|\Lsh\dhn\|^2_\OS
\le \Eho+\dt^2\|\d_{1,h}\|^2_{\rm s}+\frac{\dt}{2\rs}\|\Lsh\d_{0,h}\|^2_\OS.
\end{equation}
\item Scheme with $r=2$. Let $\dt$ and $h_{\rm s}$ be such that there exist $\alpha>0$ such that
\begin{equation}
\label{eq:CFL}\left\{ 
 \begin{aligned}
 &\dt \le \alpha \left(\frac{\rs}{ C_{\rm I} }\right)^\frac23 h_{\rm s}^\frac43,\\
 & 2 \tau \alpha^3 < 1, 
 \end{aligned}\right. 
\end{equation}
then, for $n\ge1$, we have 
\begin{equation}
\label{eq:enr2}
\Ehn+\Dhn\le \mathrm{exp}\left(\frac{2\gamma t_n}{1-2\dt\gamma}\right)\Eho.
\end{equation}
\end{itemize}
\end{thm}
\begin{proof}
By taking $\v=\uhn$, $q=\phn$, $\w=\ddhnn$ and $\mmu=-\lhn$
in~\eqref{eq:split1}, and using the well known equality
$2(a-b,a)=(a^2-b^2+(a-b)^2)$, we have 
\begin{equation}
\label{eq:p1}
\aligned
&\frac{\rf}2\left(\|\uhn\|^2_\OO-\|\uhnu\|^2_\OO+\|\uhn-\uhnu\|^2_\OO\right)
+2\dt\mu\|\eeps(\uhn)\|^2_\OO+\dt|\phn|^2_{\sh}\\
&\quad
+\rs\big(\ddhnn-\ddhnu,\ddhnn\big)_\Sigma=-\dt\as\big(\dhnstar,\ddhnn\big).
\endaligned
\end{equation} 
%
%In the case $r=0$, we take $\z=\ddhn$ in the second step of the
%scheme~\eqref{eq:split2}. Moreover, since $\dhn$ and $\ddhn$ belong to $\W_h$,
%then $\pd\dhn=\ddhn$, so that it holds
%%
%\[
%\aligned
%&\frac{\rf}2\left(\|\ddhn\|^2_\OS-\|\ddhnn\|^2_\OS+\|\ddhn-\ddhnn\|^2_\OS\right)
%\\
%&\qquad+\frac12\left(\|\dhn\|^2_{\rm s}-\|\dhnu\|^2_{\rm s}+\|\dhn-\dhnu\|^2_{\rm s}\right)=0.
%\endaligned
%\]
%%
%We add the last equality to~\eqref{eq:p1} and we obtain
%%
%\[
%\aligned
%&\frac{\rf}2\left(\|\uhn\|^2_\OO-\|\uhnu\|^2_\OO+\|\uhn-\uhnu\|^2_\OO\right)
%+2\dt\mu\|\eeps(\uhn)\|^2_\OO+\dt|\phn|^2_{\sh}\\
%&\quad+\frac{\rs}2\left(\|\ddhn\|^2_\OS-\|\ddhnu\|^2_\OS
%+\|\ddhn-\ddhnn\|^2_\OS+\|\ddhnn-\ddhnu\|^2_\OS\right)\\
%&\quad+\frac12\left(\|\dhn\|^2_{\rm s}-\|\dhnu\|^2_{\rm s}+\|\dhn-\dhnu\|^2_{\rm s}\right)=0.
%\endaligned
%\]
%Summing on $n$, we arrive to the desired energy estimate~\eqref{eq:enr0}.
%\\
On the other hand, by testing \eqref{eq:end-of-step-vel} with $\w=\ddhnn$  and by adding the
resulting equation to~\eqref{eq:p1}, we get 
\[
\aligned
&\frac{\rf}2\left(\|\uhn\|^2_\OO-\|\uhnu\|^2_\OO+\|\uhn-\uhnu\|^2_\OO\right)
+2\dt\mu\|\eeps(\uhn)\|^2_\OO+\dt|\phn|^2_{\sh}\\
&\quad+\rs\big(\ddhn-\ddhnu,\ddhnn\big)_\Sigma+\dt\as(\dhn,\ddhnn)=0.
\endaligned
\]
By introducing in the above equation the characterization of $\ddhnn$ given
in~\eqref{eq:dn12} yields 
\begin{equation}
\label{eq:p2}
\aligned
&\frac{\rf}2\left(\|\uhn\|^2_\OO-\|\uhnu\|^2_\OO+\|\uhn-\uhnu\|^2_\OO\right)
+2\dt\mu\|\eeps(\uhn)\|^2_\OO+\dt|\phn|^2_{\sh}\\
&\quad
+\frac{\rs}2\left(\|\ddhn\|^2_\OS-\|\ddhnu\|^2_\OS+\|\ddhn-\ddhnu\|^2_\OS\right)\\
&\quad+\frac12\left(\|\dhn\|^2_{\rm s}-\|\dhnu\|^2_{\rm s}+\|\dhn-\dhnu\|^2_{\rm s}\right)+T_1+T_2=0,
\endaligned
\end{equation}
with
\[
T_1:=\dt\left(\ddhn-\ddhnu,\Lsh(\dhn-\dhnstar)\right)_\Sigma\quad
T_2:=\frac{\dt^2}{\rs}\as(\dhn,\Lsh(\dhn-\dhnstar)).
\]
We estimate this terms  as in \cite[Theorem 1]{Fernandez2013}, by treating each case 
of extrapolation separately.

\paragraph{Case $r=1$.} We have $\dhnstar=\dhnu$, so that
$\Lsh(\dhn-\dhnstar)=\Lsh(\dhn-\dhnu)=\dt\Lsh\ddhn$. By using the definition of
the discrete operator $\Lsh$ we get the following relations for $T_1$ and
$T_2$:
\[
\aligned
T_1&=\dt^2\left(\ddhn-\ddhnu,\Lsh\ddhn\right)_\Sigma
=\dt^2\as\left(\ddhn-\ddhnu,\ddhn\right)_\Sigma\\
&=\frac{\dt}2\left(\|\ddhn\|^2_{\rm s}-\|\ddhnu\|^2_{\rm s}+\|\ddhn-\ddhnu\|^2_{\rm s}\right),\\
T_2&=\frac{\dt^2}{\rs}\left(\Lsh\dhn,\Lsh(\dhn-\dhnstar)\right)_\Sigma\\
&=\frac{\dt^2}{\rs}
\left(\|\Lsh\dhn\|^2_\OS-\|\Lsh\dhnu\|^2_\OS+\|\Lsh(\dhn-\dhnu)\|^2_\OS\right).
\endaligned
\]
By inserting these equalities into~\eqref{eq:p2} and summing over $n$, we
obtain~\eqref{eq:enr1}.

\paragraph{Case $r=2$.} We have $\dhnstar=\dhnu+\dt\ddhnu$, which  yields 
\[
\Lsh(\dhn-\dhnstar)=\Lsh(\dhn-\dhnu-\dt\ddhnu)=
\dt\Lsh(\ddhn-\ddhnu).
\]
Substituting the last relation in $T_1$ and $T_2$ gives
\[
\aligned
T_1&=\dt^2\left(\ddhn-\ddhnu,\Lsh(\ddhn-\ddhnu)\right)_\Sigma\\
&=\dt^2\as(\ddhn-\ddhnu,\ddhn-\ddhnu)=\dt^2\|\ddhn-\ddhnu\|^2_{\rm s},\\
T_2&=\frac{\dt^3}{\rs}\as(\dhn,\Lsh(\ddhn-\ddhnu))
=\frac{\dt^3}{\rs}\as(\Lsh\dhn,\ddhn-\ddhnu)\\
&\ge-\frac{\dt^3}{\rs}\|\Lsh\dhn\|_{\rm s}\|\ddhn-\ddhnu\|_{\rm s}\\
&\ge-\frac{\dt^4C_{\rm I}^2}{(\rs)^2h_{\rm s}^4}\|\dhn\|^2_{\rm s}-\dt^2\|\ddhn-\ddhnu\|^2_{\rm s}\\
&\ge- \tau \alpha^3\|\dhn\|^2_{\rm s}-\dt^2\|\ddhn-\ddhnu\|^2_{\rm s}.
\endaligned
\]
In the two last bounds of $T_2$,  
the inverse estimate~\eqref{eq:inverse}, the Young's inequality and \eqref{eq:CFL}$_1$ were used.
By inserting these expression into~\eqref{eq:p2} and by summing over $n$, we get
\[
\aligned
&\frac{\rf}2\|\uhn\|^2_\OO+\frac{\rs}2\|\ddhn\|^2_\OS+\frac12\|\dhn\|^2_{\rm s}
+\dt\sum_{m=1}^n\left(2\mu\|\eeps(\u_h^m)\|^2_\OO+|p_h^m|^2_{\sh}\right)\\
&\quad
\le\frac{\rf}2\|\u_{0,h}\|^2_\OO+\frac{\rs}2\|\d_{1,h}\|^2_\OS
+\frac12\|\d_{0,h}\|^2_{\rm s}
+\dt\sum_{m=1}^n\alpha^3\|\d_h^m\|^2_{\rm s}.
\endaligned
\]
Finally, the estimate~\eqref{eq:enr2} follows by applying  the discrete Gronwall's
Lemma~\ref{le:Gronwall}  with  $\gamma_m:=2\alpha^3$ and by assuming that \eqref{eq:CFL}$_2$ holds.
\end{proof}
\subsection{Error estimates for a linear model problem}
\label{se:error}
This section is devoted to the   convergence analysis of Algorithm~\ref{pb:split} by assuming that the structure undergoes infinitesimal displacements. 
We can hence identify the current configuration with the reference one. Therefore the
terms in Problem~\ref{pb:weak} and in Algorithm~\ref{pb:split} which contain the
composition of a function $v$ with the mappings $\pphi$ and $\pphi_h^{n-1}$, respectively,  will be written simply
as $v|_\Sigma$ instead of $v\circ\pphi$. Moreover, in order to simplify the
presentation, we drop out the non-linear convective term in the fluid and { we assume that the
immersed structure is represented by a closed polygonal line or surface (see Fig.~\ref{fig:domain})}.  As a
consequence the discrete spaces $\Wh$ and $\Lh$ coincide, 
\begin{equation}\label{eq:wl}
\Wh = \Lh.
\end{equation}

Since the pressure results to be discontinuous across the structure, we assume
that the solution enjoys the following regularity properties for $0<\ell<1/2$
and $0<m\le1$:
\begin{equation} 
\label{eq:regu}
\aligned
& \u\in(H^1(0,T;H^{1+\ell}(\Omega)))^d,
&&\partial_{tt}\u\in(L^2(0,T;L^2(\Omega)))^d,\\
& p\in H^1(0,T;H^\ell(\Omega)), &&\llambda\in H^1(0,T;H^{\ell-1/2}(\Sigma))^d,\\
& \d\in(H^1(0,T;H^{1+m}(\Sigma)))^d, && {  \L\d\in (L^\infty(0,T;L^2(\Sigma)))^d},\\
&\dd\in(H^1(0,T;H^{1+m}(\Sigma)))^d,
&&\partial_{tt}\dd\in(L^2(0,T;L^2(\Sigma)))^d.
\endaligned
\end{equation}
We introduce the projection operators which will be used in the proof of the
error estimates together with some approximation results.
Let $\PV:\Huo\times\Ldo\to\Vh$ and $\PQ:\Huo\times\Ldo\to\Qh$ be the
\emph{Stokes projection} operators
which to any pair $(\u,p)\in\Huo\times\Ldo$ associate the solution
$(\PV(\u,p),\PQ(\u,p))\in\Vh\times\Qh$ of the following discrete Stokes
equations
\begin{equation}
\label{eq:Stokesproj}
\af_h\big((\PV(\u,p),\PQ(\u,p)),(\v,q)\big)=\af((\u,p),(\v,q))\quad  \forall(\v,q)\in\Vh\times Q_h.
\end{equation} 
Exploiting carefully the stabilization term appearing in the Stokes equations
discretized by the stabilized $\mathbb{P}_1/\mathbb{P}_1$ elements, one can extend the standard
error estimates to the case of non smooth pressure and velocity, as follows:
\begin{equation}
\label{eq:approxStokes}
\|\u-\PV(\u,p)\|_{1,\Omega}+\|p-\PQ(\u,p)\|_{\OO}\le
Ch_{\rm f}^\ell\left(\|\u\|_{{1+\ell},\Omega}+\|p\|_{\ell,\Omega}\right).
\end{equation}
Moreover, assuming that the domain $\Omega$ is convex, by standard duality
argument one can obtain the estimate in the $L^2$-norm for the velocity, namely
\begin{equation}
\label{eq:approxStokesL2}
\|\u-\PV(\u,p)\|_{\OO}\le
Ch_{\rm f}^{1+\ell}\left(\|\u\|_{{1+\ell},\Omega}+\|p\|_{\ell,\Omega}\right).
\end{equation}
We denote by $\PW:\W\to\Wh$ the elliptic projection operator associated to the
bilinear form $\as$ as follows: for any $\d\in\W$, $\PW\d\in\Wh$ with
\begin{equation}
\label{eq:projW}
\as(\PW\d,\w_h)=\as(\d,\w_h)\quad\forall\w_h\in\Wh.
\end{equation}
Since $\as$ is assumed to be coercive on $\W$ the following approximation
estimate holds true
\begin{equation}
\label{eq:approxW}
\|\d-\PW\d\|_{\rm s}\le Ch_{\rm s}^m\|\d\|_{1+m,\Sigma}.
\end{equation}
At the end, we introduce the projection operator $\PL:\LL\to\Lh$
for the Lagrange multiplier as follows:
\begin{equation}
\label{eq:projL}
\c(\PL\llambda,\w_h)=\c(\llambda,\w_h)\quad\forall\w_h\in\Wh.
\end{equation}
We observe that we have used the same discrete space for $\Lh$ and $\Wh$ and
that for smooth functions the bilinear form $\c$ can be seen as the scalar
product in $L^2(\Sigma)$. Then we have the following approximation property.
\begin{lem}
\label{le:PL}
Assume that $\triaB$ be quasi-uniform. We have 
\begin{equation}
\label{eq:PL}
\|\llambda-\PL\llambda\|_{\LL}\le Ch_{\rm s}^\ell\|\llambda\|_{\ell-\frac12,\Sigma}
\end{equation}
for any $\llambda \in H^{\ell-\frac12}(\Sigma)^d$.
\end{lem}
\begin{proof}
In the following we shall use the $L^2$-projection $P_0$ onto $\Wh$ defined by
\[
(\w-P_0\w,\z)_\Sigma=0\qquad\forall\z\in\Wh.
\]
By definition of the norm in the space $\LL$ and using \eqref{eq:wl},
we have
\begin{equation}
\label{eq:approxL}
\aligned
\|\llambda-\PL\llambda\|_{\LL}&=
\sup_{\z\in  H^{\frac12}(\Sigma)^d }\frac{\c(\llambda-\PL\llambda,\z)}{\|\z\|_{\HhalfB}}\\
&=\sup_{\z\in H^{\frac12}(\Sigma)^d } \frac{\c(\llambda-\PL\llambda,\z-P_0\z)}{\|\z\|_{\HhalfB}}\\
&=\sup_{\z\in H^{\frac12}(\Sigma)^d}\frac{\c(\llambda,\z-P_0\z)-(\PL\llambda,\z-P_0\z)_\Sigma}
{\|\z\|_{\HhalfB}}\\
&=\sup_{\w\in H^{\frac12}(\Sigma)^d}\frac{\c(\llambda,\w-P_0\w)}{\|\w\|_{\HhalfB}}\\
&\le\sup_{\w\in H^{\frac12}(\Sigma)^d}
\frac{\|\llambda\|_{\ell-\frac12,\Sigma}\|\w-P_0\w\|_{\frac12-\ell,\Sigma}}
{\|\w\|_{\HhalfB}}.
\endaligned
\end{equation}
It remains to bound $\|\w-P_0\w\|_{H^{1/2-\ell}(\Sigma)^d}$.
Since the mesh is quasi uniform, we observe that the $L^2$-projection is stable
in $H^1(\Sigma)$, see~\cite{ABGLR} and the references quoted therein which can
weaken the requirement of a quasi-uniform mesh. 
Therefore by application of
interpolation operator theory (see for example~\cite{BrennerScott}) $P_0$ is
stable also in $H^{1/2}(\Sigma)$, 
so that there exists a constant $c_0$ such that
\[
\|P_0\w\|_{\HhalfB}\le c_0\|\w\|_{\HhalfB}.
\]
This implies the following error estimates
\[
\aligned
&\|\w-P_0\w\|_{\OS}\le Ch^{\frac12}\|\w\|_{\HhalfB}\\
&\|\w-P_0\w\|_{\HhalfB}\le (1+c_0)\|\w\|_{\HhalfB}.
\endaligned
\]
%{@Lucia: Why stability in $H^{\frac12}$ it suffices to have hat $P_0$ is stable in $H^{1/2-\ell}$, right ?}

%
Applying again the interpolation operator theory, we arrive at the desired
estimate
\[
\|\w-P_0\w\|_{\frac12-\ell,\Sigma}\le Ch^\ell\|\w\|_{\HhalfB}
\]
and this inserted in~\eqref{eq:approxL} concludes the proof.
\end{proof}
The following auxiliary result provides an estimate of the error between the
time derivative and the backward finite difference approximation.
\begin{lem}
\label{le:timeder}
Let $X$ be a real Hilbert space endowed with the norm $\|\cdot\|_X$.
Then for all $\v\in H^2(0,T;X)$ we have
\begin{equation}
\label{eq:bfd1}
\dt\|\pd\v^n-\pt\v^n\|_X\le 
\dt^{\frac32}\|\partial_{tt}\v\|_{L^2(t_{n-1},t_n;X)}.
\end{equation}
Moreover, for all  $\v\in H^1(0,T;X)$ it holds true
\begin{equation}
\label{eq:bfd2}
\dt\|\pd\v^{n}\|_X\le \dt^{\frac12} \|\pt\v\|_{L^{2}(t_{n-1},t_{n};X)}.
\end{equation}
\end{lem}
By subtracting equations~\eqref{eq:split1} and~\eqref{eq:split2}
from~\eqref{eq:weak} we obtain the error equations.
\begin{equation}
\label{eq:eqerr}\left\{
\aligned
&\rf(\pt\un-\pd\uhn,\v)+\af\big((\un-\uhn,\pn-\phn),(v,q)\big)+\c(\ln-\lhn,\v|_\Sigma)&\\
&\quad+ \c(\mmu,(\un-\uhn)|_\Sigma-(\ddn-\ddhnn))-\sh(\phn,q)=0\\
&\hspace{6cm} \forall(\v,q,\mmu)\in\V_h\times Q_h\times \Lh,\\
&\rs(\pt\ddn-\pd\ddhn,\w)_\Sigma+\as(\dn-\dhn,\w)-\c(\ln-\lhn,\w)=0
\ \forall\w\in\W_h,\\
&\pt\dn-\pd\dhn=\ddn-\ddhn.
\endaligned\right.
\end{equation}
In order to simplify the writing, we introduce some notation.\\
Given $(\un, \pn, \dn, \ddn, \ln)\in\Huo\times\Ldo\times\W\times\W\times\LL$,
we set
\[
\unp:=\PV(\un,\pn),\ \pnp:=\PQ(\un,\pn),\ \dnp:=\PW\dn,\ \ddnp:=\PW\ddn,\
\lnp=\PL\ln
\]
and we split the errors as follows:
\begin{equation}
\label{eq:errors}
\aligned
\un-\uhn& =\tnp+\tnh,  \quad\tnp:=\un-\unp,\ \tnh:=\unp-\uhn,\\
\pn-\phn&=\phinp+\phinh,  \quad\phinp:=\pn-\pnp,\ \phinh:=\pnp-\phn,\\
\dn-\dhn&=\xinp+\xinh,  \quad\xinp:=\dn-\dnp,\ \xinh:=\dnp-\dhn,\\
\ddn-\ddhn&=\dxinp+\dxinh,  \quad\dxinp:=\ddn-\ddnp,\ \dxinh:=\ddnp-\ddhn,\\
\ln-\lhn&=\omnp+\omnh,  \quad\omnp:=\ln-\lnp,\ \omnh:=\lnp-\lhn.
\endaligned
\end{equation}
In the next lemma we provide an estimate of $\omnh$ in terms of the other
errors.
\begin{lem}
\label{le:errl}
Let us assume that $\Omega$ is convex and that the mesh $\S_h$ is
quasi-uniform. If $h_{\rm  f}/h_{\rm s}$ is sufficiently small, we have
\begin{equation}
\label{eq:errl}
\aligned
\|\omnh\|_{\LL}&\le C\big(\|\omnp\|_{\LL}+
\dt^{\frac12}\|\partial_{tt}\u\|_{L^2(t_{n-1},t_n;L^2(\Omega)^d)}\\
&+\|\pd(\tnp+\tnh)\|_{\OO}+\|\eeps(\tnh)\|_{\OO}+|\phinh|_{\sh}\big)
\endaligned
\end{equation}
\end{lem}
\begin{proof}
In~\cite[Prop.~13]{BGarXiv} the following inf-sup condition has been proved:
there exists a positive constant $\beta$ such that 
\[
\beta\|\omnh\|_{\LL}\le\sup_{\v\in\V_0}\frac{\c(\omnh,\v|_\Sigma)}{\|\v\|_{1,\Omega}},
\]
where $\V_0$ denotes the subspace of $\Huo$ made of divergence free
functions. Therefore, there exists $\overline\v\in\V_0$ such that
\[\c(\omnh,\overline\v|_\Sigma)\ge
\beta\|\omnh\|_{\LL}\|\overline\v\|_{1,\Omega},\quad \|\overline\v\|_{1,\Omega}=\|\omnh\|_{\LL}.
\] 
%
%We construct an approximation of $\overline\v$ as follows. The pair
%$(\overline\v,\overline p)\in\Huo\times\Ldo$ is the solution of the following
%Stokes problem
%%
%\[
%\aligned
%&(\nabla\overline\v,\nabla\v)-(\div\v,\overline p)=(\nabla\overline\v,\nabla\v)
%&&\quad\forall\v\in\Huo,\\
%&(\div\overline\v,q)=0&&\quad\forall q\in\Ldo,
%\endaligned
%\]
%%
%with $\overline p=0$.
Let $(\overline\v_h,\overline p_h)\in\Vh\times\Qh$ be the
solution of the associated discrete problem
\begin{equation}
\label{eq:auxh}
a^{\rm f}_h\big( ( \overline \v_h,\overline p_h),(\v,q)\big) =a^{\rm f}\big( (\overline\v, 0),(\v,q)\big)
 \quad\forall (\v,q) \in\Vh \times \Qh.
\end{equation}
The following bounds thus hold true, by taking into account that $\Omega$ is
convex, 
\begin{equation}
\label{eq:auxprop}
\aligned
\|\overline\v_h\|_{1,\Omega}+\|\overline p_h\|_{\OO}+|\overline p_h|_{s_h} & \le
C\|\overline\v\|_{1,\Omega},\\
\|\overline\v-\overline\v_h\|_{\OO} & \le Ch_{\rm f}\|\overline\v\|_{1,\Omega},\\
\|\overline\v_h\|_{1,\Omega}& \le\|\overline\v\|_{1,\Omega}.
\endaligned
\end{equation}
Hence, we have
\begin{equation}
\label{eq:trepez}
\aligned
\beta\|\omnh\|_{\LL}\|\overline\v\|_1&\le\c(\omnh,\overline\v|_\Sigma)
=\c(\omnh,\overline\v|_\Sigma-\overline\v_h|_\Sigma)+
\c(\omnh,\overline\v_h|_\Sigma)\\
&=\c(\omnh,\overline\v|_\Sigma-\overline\v_h|_\Sigma)+
\c(\lnp-\ln,\overline\v_h|_\Sigma)+
\c(\ln-\lhn,\overline\v_h|_\Sigma).
\endaligned
\end{equation}
We bound the three terms on the right hand side separately.
An inverse inequality, trace theorem and the error estimates above imply
\[
\c(\omnh,\overline\v|_\Sigma-\overline\v_h|_\Sigma)\le 
C\left(\frac{h_{\rm f}}{h_{\rm s}}\right)^{\frac12}\|\omnh\|_{\LL}\|\overline\v\|_{1,\Omega}.
\]
For the second term we use Lemma~\ref{le:PL} as follows
\[
\c(\lnp-\ln,\overline\v_h|_\Sigma)\le
C\|\omnp\|_{\LL}\|\overline\v\|_{1,\Omega}.
\]
We use the first equation in~\eqref{eq:eqerr}, the definition of the Stokes
projection operator~\eqref{eq:Stokesproj} and~\eqref{eq:auxh} to estimate the
last term in~\eqref{eq:trepez}, namely 
\[
\aligned
\c(\ln-\lhn,\overline\v_h|_\Sigma)&=
-\rf(\pt\un-\pd\uhn,\overline\v_h)-2\mu(\eeps(\un-\uhn),\eeps(\overline\v_h))\\
&\qquad+(\div\overline\v_h,\pn-\phn)\\
&=-\rf((\pt-\pd)\un,\overline\v_h)-\rf(\pd(\tnp+\tnh),\overline\v_h)\\
&\qquad-2\mu (\eeps(\tnh),\eeps(\overline\v_h))-\sh(\overline p_h,\phinh)\\
&\le C\|\overline\v\|_1
\big(\dt^{\frac12}\|\partial_{tt}\u\|_{L^2(t_{n-1},t_n;L^2(\Omega)^d)}\\
&\qquad+\|\pd(\tnp+\tnh)\|_{\OO}+\|\eeps(\tnh)\|_{\OO}+|\phinh|_{\sh}\big).
\endaligned
\]
Putting together the last inequalities in~\eqref{eq:trepez} and taking into
account~\eqref{eq:auxprop}, we obtain
\[
\aligned
\beta\|\omnh\|_{\LL}&\le C\left(\frac{h_{\rm f}}{h_{\rm s}}\right)^{1/2}\|\omnh\|_{\LL}
+C\Big(\|\omnp\|_{\LL}
+\dt^{1/2}\|\partial_{tt}\u\|_{L^2(t_{n-1},t_n;L^2(\Omega)^d)}\\
&\quad+\|\pd(\tnp+\tnh)\|_{\OO}+\|\eeps(\tnh)\|_{\OO}+|\phinh|_{\sh}\Big).
\endaligned
\]
Choosing $h_{\rm f}/h_{\rm s}$ sufficiently small we get~\eqref{eq:errl}, which concludes the proof. 
\end{proof}

The solid intermediate velocity $\ddhnn$ provided in Step~1 of Algorithm~\ref{pb:split} is
actually an approximation of $\ddhn$, hence we introduce the following error
\[
\cchinh=\ddnp-\ddhnn.
\]
Hence, owing to~\eqref{eq:dn12}, we have
\begin{equation}
\label{eq:chin}
\cchinh=\dxinh-\frac{\dt}{\rs}\Lsh(\dhn-\dhnstar)\\
=\dxinh+\frac{\dt}{\rs}\Lsh(\xinh-\xinstar)-\frac{\dt}{\rs}\Lsh(\dn-\dnstar).
\end{equation}
The following theorem states the main result of this section. It provides an error bound on the discrete approximation errors.
\begin{thm}
\label{th:error}
Let $(\un,\pn,\dn,\ddn,\ln)\in\Huo\times\Ldo\times\W\times\W\times\LL$ be the
solution of Problem~\ref{pb:weak} and let
$(\uhn,\phn,\ddhnn,\ln)\in\Vh\times\Qh\times\Wh\times\Lh$ and
$(\dhn,\ddhn)\in\Wh\times\Wh$ be given by  Algorithm~\ref{pb:split}, respectively. Then, if $h_{\rm f}/h_{\rm s}$ is sufficiently small,
the following bounds hold true:
\begin{itemize}
\item Scheme with $r=1$:
\begin{equation}
\label{eq:err1}
\aligned
\rf&\|\tnh\|^2_{\OO}+\rs\|\dxinh\|^2_{0,\Sigma}+\|\xinh\|^2_{\rm s}\\
&\le C\Big(\dt^2\|\partial_{tt}\u\|^2_{L^2(0,t_n;L^2(\Omega)^d)}
+\dt^2\|\partial_{tt}\dd\|^2_{L^2(0,t_n;L^2(\Sigma)^d)}\\
&+\dt^2\|\partial_{tt}\d\|_{L^2(0,t_n;H^1(\Sigma)^d)}
+\dt^5\|\partial_{tt}\dd\|^2_{L^2(0,t_n;H^1(\Sigma)^d)}
+\|\pt\boldsymbol{\theta}_\Pi\|^2_{L^2(0,t_n;L^2(\Omega)^d)}\\
&+\|\pt\dot{\boldsymbol{\xi}}_\Pi\|^2_{L^2(0,t_n;L^2(\Sigma)^d)}
+\dt\|\pt\xi_\Pi\|^2_{L^2(0,t_n;L^2(H^1(\Sigma)^d)}\\
&+\dt^3\|\pd\dot{\boldsymbol{\xi}}_\Pi\|^2_{L^2(0,t_n;L^2(H^1(\Sigma)^d)}
+\dt^2\|\pt\L\d\|^2_{L^2(0,t_n;L^2(\Sigma)^d)}\\
&+\sum_{k=1}^n\big(\dt\|\boldsymbol{\omega}^k_\Pi\|^2_{\LL}
+\dt\|\boldsymbol{\theta}^k_\Pi\|^2_{1,\Omega}
+\dt\|\dot{\boldsymbol{\xi}}^k_\Pi\|^2_{\HhalfB}
+\dt^2\|\dot{\boldsymbol{\xi}}^k_\Pi\|^2_{\rm s}\big)
\Big).
\endaligned
\end{equation}
\item Scheme with $r=2$:
let $\dt$ such that~\eqref{eq:CFL} holds true and 
\[
\frac{\dt^3}{(\rs)^2}\le1
\]
then for $n\ge1$
\begin{equation}
\label{eq:err2}
\aligned
\rf&\|\tnh\|^2_{\OO}+\rs\|\dxinh\|^2_{0,\Sigma}+\|\xinh\|^2_{\rm s}\\
&\le C\Big(\dt^2\|\partial_{tt}\u\|^2_{L^2(0,t_n;L^2(\Omega)^d)}
+\dt^2\|\partial_{tt}\dd\|^2_{L^2(0,t_n;L^2(\Sigma)^d)}\\
&+\dt^2\|\partial_{tt}\d\|_{L^2(0,t_n;H^1(\Sigma)^d)}
+\|\pt\tnp\|^2_{L^2(0,t_n;L^2(\Omega)^d)}
+\|\pt\dxinp\|^2_{L^2(0,t_n;L^2(\Sigma)^d)}\\
&+\dt^3\|\pt\dxinp\|^2_{L^2(0,t_n;H^1(\Sigma)^d)}
+\dt^5\|\pt\L\dd\|_{L^2(0,t_n;L^2(\Sigma)^d)}\\
&+\sum_{k=1}^n\big(\dt\|\boldsymbol{\omega}^k_\Pi\|^2_{\LL}
+\dt\|\boldsymbol{\theta}^k_\Pi\|^2_{1,\Omega}
+\dt\|\dot{\boldsymbol{\xi}}^k_\Pi\|^2_{\HhalfB}\big)\Big).
\endaligned
\end{equation}
\end{itemize}
\end{thm}
\begin{proof}
By using the notation introduced in~\eqref{eq:errors} and recalling the
definitions of the projection operators introduced in~\eqref{eq:Stokesproj},
\eqref{eq:projW} and~\eqref{eq:projL}, the error equation \eqref{eq:eqerr} yields 
\begin{equation}
\label{eq:eqerrmod}\left\{ 
\aligned
&\rf(\pd\tnh,\v)+2\mu(\eeps(\tnh),\eeps(\v)-(\div\v,\phinh)+\c(\omnh,\v|_\Sigma)\\
&\quad=-\rf(\pt\un-\pd\un,\v)-\rf(\pd\tnp,\v)-\c(\omnp,\v|_\Sigma)
&&\ \forall\v\in\V_h,\\
&(\div(\tnh,q)+\sh(\phinh,q)=0&&\ \forall q\in Q_h,\\
&\c\big(\mmu,\tnh|_\Sigma-\cchinh\big)
=-\c\big(\mmu,\tnp|_\Sigma-\dxinp\big)
&&\ \forall\mmu\in\Lh,\\
&\rs(\pd\dxinh,\w)_\Sigma+\as(\xinh,\w)-\c(\omnh,\w)\\
&\qquad =-\rs(\pt\ddn-\pd\ddn,\w)_\Sigma-\rs(\pd\dxinp,\w)_\Sigma
&&\ \forall\w\in\W_h,\\
&\pd\xinh=\dxinh-\ddnp+\pd\dnp.
\endaligned\right. 
\end{equation}
We take $\v=\dt\tnh$, $q=\dt\phinh$, $\w=\dt\cchinh$, $\mmu=-\dt\omnh$ and sum the resulting expressions, so that we have
\[
\aligned
&\rf(\tnh-\tnuh,\tnh)+\dt2\mu(\eeps(\tnh),\eeps(\tnh))+\dt\sh(\phinh,\phinh)\\
&\qquad+\rs(\dxinh-\dxinuh,\cchinh)_\Sigma+\dt\as(\xinh,\cchinh)\\
&=-\dt\rf(\pt\un-\pd\un,\tnh)-\dt\rf(\pd\tnp,\tnh)-\dt\c(\omnp,\tnh|_\Sigma)\\
&\qquad-\dt\rs(\pt\ddn-\pd\ddn,\cchinh)_\Sigma
-\dt\rs(\pd\dxinp,\cchinh)_\Sigma
+\dt\c(\omnh,\tnp|_\Sigma-\dxinp).
\endaligned
\]
We observe that using~\eqref{eq:chin}, last equation in~\eqref{eq:eqerrmod}
and~\eqref{eq:projW}, we have 
\[
\as(\xinh,\cchinh)=\as(\xinh,\pd\xinh)+\as(\xinh,\pt\dn-\pd\dn).
\]
Using the well-known identity $(a-b)a=\frac12(a^2-b^2+(a-b)^2)$
and~\eqref{eq:chin}, we get
\begin{equation}
\label{eq:via}
\aligned
&\frac{\rf}2\big(\|\tnh\|^2_{\OO}-\|\tnuh\|^2_{\OO}+\|\tnh-\tnuh\|^2_{\OO}\big)
+2\dt\mu\|\eeps(\tnh)\|^2_{0,\Omega}+\dt|\phinh|_{\sh}\\
&\quad+\frac{\rs}2\big(\|\dxinh\|^2_{\OS}-\|\dxinuh\|^2_{\OS}
+\|\dxinh-\dxinuh\|^2_{\OS}\big)\\
&\quad+\frac12\big(\|\xinh\|^2_{\rm s}-\|\xinuh\|^2_{\rm s}
+\|\xinh-\xinuh\|^2_{\rm s}\big)
=\sum_{i=1}^8T_i,
\endaligned
\end{equation}
with the notations 
\begin{equation}
\label{eq:defterms}
\aligned
&T_1:=-\dt\rf(\pt\un-\pd\un,\tnh)-\dt\rf(\pd\tnp,\tnh),\\
&T_2:=-\dt\rs(\pt\ddn-\pd\ddn,\dxinh)_\Sigma-\dt\rs(\pd\dxinp,\dxinh)_\Sigma,\\
&T_3:=-\dt\as(\xinh,\pt\dn-\pd\dn),\\
&T_4:=-\dt\c(\omnp,\tnh|_\Sigma),\\
&T_5:=\dt\c(\omnh,\tnp|_\Sigma-\dxinp),\\
&T_6:=\frac{\dt^2}{\rs}\as\big(\xinh,\Lsh(\dhn-\dhnstar)\big),\\
&T_7:=\dt\big(\dxinh-\dxinuh,\Lsh(\dhn-\dhnstar)\big)_\Sigma,\\
&T_8:=\dt^2\big(\pt\ddn-\pd\ddn,\Lsh(\dhn-\dhnstar)\big)_\Sigma
+\dt^2\big(\pd\dxinp,\Lsh(\dhn-\dhnstar)\big)_\Sigma.
\endaligned
\end{equation}
We estimate the first 5 terms which do not depend on $\dhnstar$ using
Lemmas~\ref{le:timeder} and~\ref{le:errl}, which yields 
\begin{equation}
\label{eq:5T}
\aligned
&T_1\le C\rf\Big(\dt^{3/2}\|\partial_{tt}\u\|_{L^2(t_{n-1},t_n;L^2(\Omega)^d)}
+\dt^{1/2}\|\pt\tnp\|_{L^2(t_{n-1},t_n;L^2(\Omega)^d)}\Big)\|\tnh\|_{\OO},
\\
&T_2\le C\rs\Big(\dt^{3/2}\|\partial_{tt}\dd\|_{L^2(t_{n-1},t_n;L^2(\Sigma)^d)}
+\dt^{1/2}\|\pt\dxinp\|_{L^2(t_{n-1},t_n;L^2(\Sigma)^d)}\Big)\|\dxinh\|_{\OS},
\\
&T_3\le
C\dt^{3/2}\|\partial_{tt}\d\|_{L^2(t_{n-1},t_n;H^1(\Sigma)^d)}\|\xinh\|_{\rm s},
\\
&T_4\le C\dt\|\omnp\|_{\LL}\|\tnh\|_{1,\Omega},
\\
&T_5\le C\dt\|\omnh\|_{\LL}\Big(\|\tnp\|_{1,\Omega}+\|\dxinp\|_{\HhalfB}\Big).
\endaligned
\end{equation}
Using Young's inequality in~\eqref{eq:5T} and the Korn inequality           
$K\|\v\|_{1,\Omega}\le\|\eeps(\v)\|_{\OO}$ for all $\v\in\Huo$,
and adding the resulting inequalities to~\eqref{eq:via} we have
\begin{equation}
\label{eq:firststep}
\aligned
&\frac{\rf}2\big(\|\tnh\|^2_{\OO}-\|\tnuh\|^2_{\OO}+\|\tnh-\tnuh\|^2_{\OO}\big)
+2\dt\mu K^2\|\tnh\|^2_{1,\Omega}+\dt|\phinh|_{\sh}\\
&\quad+\frac{\rs}2\big(\|\dxinh\|^2_{\OS}-\|\dxinuh\|^2_{\OS}
+\|\dxinh-\dxinuh\|^2_{\OS}\big)\\
&\quad+\frac12\big(\|\xinh\|^2_{\rm s}-\|\xinuh\|^2_{\rm s}+\|\xinh-\xinuh\|^2_{\rm s}\big)\\
&\le\frac{\dt\delta_1}2\Big(\rf\|\tnh\|^2_{\OO}
+\rs\|\dxinh\|^2_{\OS}+\|\xinh\|^2_{\rm s}+\|\tnh\|^2_{1,\Omega}+|\phinh|^2_{\sh}\Big)\\
&\quad+C\Big(\dt^2\|\partial_{tt}\u\|^2_{L^2(t_{n-1},t_n;L^2(\Omega)^d)}
+\dt^2\|\partial_{tt}\dd\|^2_{L^2(t_{n-1},t_n;L^2(\Sigma)^d)}\\
&\quad +\dt^2\|\partial_{tt}\d\|_{L^2(t_{n-1},t_n;H^1(\Sigma)^d)}
+\|\pt\tnp\|^2_{L^2(t_{n-1},t_n;L^2(\Omega)^d)}\\
&\quad+\|\pt\dxinp\|^2_{L^2(t_{n-1},t_n;L^2(\Sigma)^d)}
+\dt\|\omnp\|^2_{\LL}+\dt\|\tnp\|^2_{1,\Omega}+\dt\|\dxinp\|^2_{\HhalfB}
\Big)+\sum_{i=6}^8T_i.
\endaligned
\end{equation}
For the remaining three terms $T_i$ for $i=6,7,8$ we have to take into account
the definition of $\dhnstar$.

\paragraph{Case $r=1$.}
We estimate the term $T_6$ by noting that $\dhn=\dnp-\xinh$
and using \eqref{eq:defLh}. We have, 
\begin{equation}
\label{eq:T61}
\aligned
T_6&=\frac{\dt^2}{\rs}\as\big(\xinh,\Lsh(\dhn-\dhnu)\big)\\
&=-\frac{\dt^2}{\rs}\as\big(\xinh,\Lsh(\xinh-\xinuh)\big)
+\frac{\dt^2}{\rs}\as\big(\xinh,\Lsh(\dnp-\d^{n-1}_\Pi)\big)\\
&=-\frac{\dt^2}{\rs}\big(\Lsh\xinh,\Lsh(\xinh-\xinuh)\big)_\Sigma
+\frac{\dt^2}{\rs}\big(\Lsh\xinh,\Lsh(\dnp-\d^{n-1}_\Pi)\big)_\Sigma\\
&\le-\frac12\frac{\dt^2}{\rs}\big(\|\Lsh\xinh\|^2_{\OS}
-\|\Lsh\xinuh\|^2_{\OS}+\|\Lsh(\xinh-\xinuh)\|^2_{\OS}\big)\\
&\quad+\frac{\dt^2}{\rs}\|\Lsh\xinh\|_{\OS}\|\Lsh(\dn-\d^{n-1})\|_{\OS}.
\endaligned
\end{equation}

The last equation in~\eqref{eq:eqerrmod} implies that
$\xinh-\xinuh=\dt\dxinh+\dt\dxinp-\dt(\pt\dn-\pd\dn)-\dt\pd\xinp$,
which inserted in $T_7$ gives
\begin{equation}
\label{eq:T71}
\aligned
T_7&=\dt\big(\dxinh-\dxinuh,\Lsh(\dhn-\dhnu)\big)_\Sigma\\
&=-\dt\big(\dxinh-\dxinuh,\Lsh(\xinh-\xinuh)\big)_\Sigma
+\dt\big(\dxinh-\dxinuh,\Lsh(\dnp-\d^{n-1}_\Pi)\big)_\Sigma\\
&=-\dt\as(\dxinh-\dxinuh,\xinh-\xinuh)
+\dt\big(\dxinh-\dxinuh,\Lsh(\dnp-\d^{n-1}_\Pi)\big)_\Sigma\\
&=-\dt^2\as(\dxinh-\dxinuh,\dxinh)-\dt^2\as(\dxinh-\dxinuh,\dxinp)\\
&\quad
+\dt^2\as(\dxinh-\dxinuh,\pt\dn-\pd\dn)+\dt^2\as(\dxinh-\dxinuh,\pd\xinp)\\
&\quad+\dt\big(\dxinh-\dxinuh,\Lsh(\dnp-\d^{n-1}_\Pi)\big)_\Sigma\\
&=-\frac{\dt^2}2\big(\|\dxinh\|^2_{\rm s}-\|\dxinuh\|^2_{\rm s}
+\|\dxinh-\dxinuh\|^2_{\rm s}\big)\\
&\quad+\dt^2\|\dxinh-\dxinuh\|_{\rm s}\Big(\|\dxinp\|_{\rm s}
+\dt^{1/2}\|\partial_{tt}\d\|_{L^2(t_{n-1},t_n;H^1(\Sigma)^d)}
+\|\pd\xinp\|_{\rm s}\Big)\\
&\quad+\dt\|\dxinh-\dxinuh\|_{\OS}\|\Lsh(\dn-\d^{n-1})\|_{\OS}.
\endaligned
\end{equation}
The last term can be easily bounded as follows
\begin{equation}
\label{eq:T81}
\aligned
T_8&=\dt^2\big(\pt\ddn-\pd\ddn,\Lsh(\dhn-\dhnu)\big)_\Sigma
+\dt^2\big(\pd\dxinp,\Lsh(\dhn-\dhnu)\big)_\Sigma\\
&=-\dt^2\as(\pt\ddn-\pd\ddn,\xinh-\xinuh)
-\dt^2\as(\pd\dxinp,\xinh-\xinuh)\\
&\quad+\dt^2\big(\pt\ddn-\pd\ddn,\Lsh(\dnp-\d^{n-1}_\Pi)\big)_\Sigma
+\dt^2\big(\pd\dxinp,\Lsh(\dnp-\d^{n-1}_\Pi)\big)_\Sigma\\
&\le\dt^2\|\xinh-\xinuh\|^2_{\rm s}\big(\|\pt\ddn-\pd\ddn\|_{\rm s}+\|\pd\dxinp\|_{\rm s}\big)\\
&\quad+C\|\Lsh(\dn-\d^{n-1})\|_{\OS}
\big(\|\pt\ddn-\pd\ddn\|_{\OS}+\|\pd\dxinp\|_{\OS}\big).
\endaligned
\end{equation}
We estimate $\|\Lsh(\dn-\d^{n-1})\|_{\OS}$ on the right hand side
of~\eqref{eq:T61}-\eqref{eq:T81} using~\eqref{eq:stabLh} and~\eqref{eq:bfd2}
as follows
\begin{equation}
\label{eq:Lhdn}
\aligned
\|\Lsh(\dn-\d^{n-1})\|_{\OS}&\le C\|\L(\dn-\d^{n-1})\|_{\OS}\le C\dt^{1/2}\|\pt\L\d\|_{L^2(t_{n-1},t_n;L^2(\Sigma)^d)}.
\endaligned
\end{equation} 
We use Young's inequality in~\eqref{eq:T61}-\eqref{eq:T81}  and insert the
resulting relations into~\eqref{eq:firststep},  which yields 
\[
\aligned
&\frac{\rf}2\big(\|\tnh\|^2_{\OO}-\|\tnuh\|^2_{\OO}+\|\tnh-\tnuh\|^2_{\OO}\big)
+2\dt\mu K^2\|\tnh\|^2_{1,\Omega}+\dt|\phinh|_{\sh}\\
&\quad+\frac{\rs}2\big(\|\dxinh\|^2_{\OS}-\|\dxinuh\|^2_{\OS}
+\|\dxinh-\dxinuh\|^2_{\OS}\big)\\
&\quad+\frac12\big(\|\xinh\|^2_{\rm s}-\|\xinuh\|^2_{\rm s}+\|\xinh-\xinuh\|^2_{\rm s}\big)\\
&\quad+\frac12\frac{\dt^2}{\rs}\big(\|\Lsh\xinh\|^2_{\OS}
-\|\Lsh\xinuh\|^2_{\OS}+\|\Lsh(\xinh-\xinuh)\|^2_{\OS}\big)\\
&\quad+\frac{\dt^2}2\big(\|\dxinh\|^2_{\rm s}-\|\dxinuh\|^2_{\rm s}
+\|\dxinh-\dxinuh\|^2_{\rm s}\big)\\
&\le\frac{\dt\delta_1}2\Big(\rf\|\tnh\|^2_{\OO}
+\rs\|\dxinh\|^2_{\OS}+\|\xinh\|^2_{\rm s}+\|\tnh\|^2_{1,\Omega}+|\phinh|^2_{\sh}
+\frac{\dt^2}{\rs}\|\Lsh\xinh\|^2_{\OS}\Big)\\
&\quad+\frac{\delta_1}2\Big(\dt^2\|\dxinh-\dxinuh\|^2_{\rm s}
+\rs\|\dxinh-\dxinuh\|^2_{\OS}+\|\xinh-\xinuh\|^2_{\rm s}\Big)\\
&\quad+C\Big(\dt^2\|\partial_{tt}\u\|^2_{L^2(t_{n-1},t_n;L^2(\Omega)^d)}
+\dt^2\|\partial_{tt}\dd\|^2_{L^2(t_{n-1},t_n;L^2(\Sigma)^d)}\\
&\quad+\dt^2\|\partial_{tt}\d\|_{L^2(t_{n-1},t_n;H^1(\Sigma)^d)}
+\dt^5\|\partial_{tt}\dd\|^2_{L^2(t_{n-1},t_n;H^1(\Sigma)^d)}\\
&\quad+\|\pt\boldsymbol{\theta}_\Pi\|^2_{L^2(t_{n-1},t_n;L^2(\Omega)^d)}
+\|\pt\dot{\boldsymbol{\xi}}_\Pi\|^2_{L^2(t_{n-1},t_n;L^2(\Sigma)^d)}\\
&\quad+\dt\|\pt\xi_\Pi\|^2_{L^2(t_{n-1},t_n;L^2(H^1(\Sigma)^d)}
+\dt^3\|\pd\dot{\boldsymbol{\xi}}_\Pi\|^2_{L^2(t_{n-1},t_n;L^2(H^1(\Sigma)^d)}
+\dt\|\omnp\|^2_{\LL}\\
&\quad+\dt\|\tnp\|^2_{1,\Omega}+\dt\|\dxinp\|^2_{\HhalfB}
+\dt^2\|\pt\L\d\|^2_{L^2(t_{n-1},t_n;L^2(\Sigma)^d)}
+\dt^2\|\dxinp\|^2_{\rm s}
\Big).
\endaligned
\]
The error estimate~\eqref{eq:err1} follows by choosing $\delta_1=1/2$, so that
the terms in the second bracket on the right hand side can be absorbded into
the left hand side, then we sum over $n$ and apply Lemma~\ref{le:Gronwall}.

\paragraph{Case $r=2$.}
Since $\pd\dhn=\ddhn$, we have that 
\[
\aligned
\dhn-\dhnstar&=\dhn-\dhnu-\dt\ddhnu=\dt(\pd\dhn-\ddhn)+\dt(\ddhn-\ddhnu)\\
&=\dt(\ddnp-\ddnup)-\dt(\dxinh-\dxinuh)
\endaligned
\]
As done for the first two schemes we analyze the three terms $T_i$ with
$i=6,7,8$.
\begin{equation}
\label{eq:T62}
\aligned
T_6=&\frac{\dt^3}{\rs}\as(\xinh,\ddnp-\ddnup)
-\frac{\dt^3}{\rs}\as(\xinh,\dxinh-\dxinuh)\\
&\le \frac{\dt^3}{\rs}\|\xinh\|_{\rm s}\Big(\|\ddnp-\ddnup\|_{\rm s}
+\|\dxinh-\dxinuh\|_{\rm s}\Big)
\endaligned
\end{equation}
Taking into account the definition~\eqref{eq:defLh}, we can write $T_7$ and
$T_8$ as follows
\begin{equation}
\label{eq:T72}
\aligned
T_7&=\dt^2(\dxinh-\dxinuh,\Lsh(\ddnp-\ddnup))_\Sigma
-\dt^2\as(\dxinh-\dxinuh,\dxinh-\dxinuh)\\
&\le\dt^2\|\dxinh-\dxinuh\|_{\OS}\|\Lsh(\ddnp-\ddnup)\|_{\OS}
-\dt^2\|\dxinh-\dxinuh\|^2_{\OS},
\endaligned
\end{equation}
\begin{equation}
\label{eq:T82}
\aligned
T_8&=\dt^3\big(\pt\ddn-\pd\ddn,\Lsh(\ddnp-\ddnup)\big)_\Sigma
+\dt^3\big(\pd\dxinp,\Lsh(\ddnp-\ddnup)\big)_\Sigma\\
&\quad-\dt^3\as(\pt\ddn-\pd\ddn,\dxinh-\dxinuh)
-\dt^3\as(\pd\dxinp,\dxinh-\dxinuh)\\
&\le\dt^3\big(\|\pt\ddn-\pd\ddn\|_{\OS}+\|\pd\dxinp\|_{\OS}\big)
\|\Lsh(\ddnp-\ddnup)\|_{\OS}\\
&\quad+\dt^3\big(\|\pt\ddn-\pd\ddn\|_{\rm s}+\|\pd\dxinp\|_{\rm s}\big)\|\dxinh-\dxinuh\|_{\rm s}.
\endaligned
\end{equation}
Adding~\eqref{eq:T62}-\eqref{eq:T82} to~\eqref{eq:firststep} and using Young's
inequality, we obtain
\[
\aligned
&\frac{\rf}2\big(\|\tnh\|^2_{\OO}-\|\tnuh\|^2_{\OO}+\|\tnh-\tnuh\|^2_{\OO}\big)
+2\dt\mu K^2\|\tnh\|^2_{1,\Omega}+\dt|\phinh|_{\sh}\\
&\quad+\frac{\rs}2\big(\|\dxinh\|^2_{\OS}-\|\dxinuh\|^2_{\OS}
+\|\dxinh-\dxinuh\|^2_{\OS}\big)\\
&\quad+\frac12\big(\|\xinh\|^2_{\rm s}-\|\xinuh\|^2_{\rm s}+\|\xinh-\xinuh\|^2_{\rm s}\big)
+\frac{\dt^2}2\|\dxinh-\dxinuh\|^2_{\rm s}\\
&\le\frac{\dt\delta_1}2\Big(\rf\|\tnh\|^2_{\OO}
+\rs\|\dxinh\|^2_{\OS}+\|\xinh\|^2_{\rm s}+\|\tnh\|^2_{1,\Omega}+|\phinh|^2_{\sh}\Big)\\
&\quad+\frac{\dt^5}{\delta_1(\rs)^2}\|\dxinh-\dxinuh\|^2_{\rm s}
+\frac{\delta_1}2\rs\|\dxinh-\dxinuh\|^2_{\OS}
\\
&\quad+C\Big(\dt^2\|\partial_{tt}\u\|^2_{L^2(t_{n-1},t_n;L^2(\Omega)^d)}
+\dt^2\|\partial_{tt}\dd\|^2_{L^2(t_{n-1},t_n;L^2(\Sigma)^d)}\\
&\quad+\dt^2\|\partial_{tt}\d\|_{L^2(t_{n-1},t_n;H^1(\Sigma)^d)}
+\|\pt\tnp\|^2_{L^2(t_{n-1},t_n;L^2(\Omega)^d)}\\
&\quad+\|\pt\dxinp\|^2_{L^2(t_{n-1},t_n;L^2(\Sigma)^d)}
+\dt^3\|\pt\dxinp\|^2_{L^2(t_{n-1},t_n;H^1(\Sigma)^d)}\\
&\quad+\dt\|\omnp\|^2_{\LL}+\dt\|\tnp\|^2_{1,\Omega}+\dt\|\dxinp\|^2_{\HhalfB}\\
&\quad+\dt^5\|\ddn-\dd^{n-1}\|^2_{\rm s}+\dt^4\|\Lsh(\ddn-\dd^{n-1})\|^2_{\OS}
\Big)
\endaligned
\]
For the last two terms in the above inequality, we use again~\eqref{eq:bfd2},
so that we have
\[
\aligned
&\|\ddn-\dd^{n-1}\|_{\rm s}\le \dt^{1/2}\|\pt\dd\|_{L^2(t_{n-1},t_n;H^1(\Sigma)^d)}
\le\dt^{1/2}\|\partial_{tt}\d\|_{L^2(t_{n-1},t_n;H^1(\Sigma)^d)},\\
&\|\Lsh(\ddn-\dd^{n-1})\|_{\OS}\le C\|\L(\ddn-\dd^{n-1})\|_{\OS}
\le C\dt^{1/2}\|\pt\L\dd\|_{L^2(t_{n-1},t_n;L^2(\Sigma)^d)}.
\endaligned
\]
We choose $\delta_1=1/2$ and $\dt$ such that $\frac{\dt^3}{(\rs)^2}\le1$, then
the application of Lemma~\ref{le:Gronwall} yields the estimate~\eqref{eq:err2}.
\end{proof}

\section{Numerical experiments}\label{sect:exp}

In this section, we perform numerical tests to check numerically the
performances of the schemes reported in  Algorithm~\ref{pb:split}. In particular, we shall
consider stability and convergence, and compare the behavior of the proposed splitting 
schemes with that of the monolithic one.
All the numerical tests are performed using the classical 2D benchmark problem of
an ellipsoidal structure that evolves to a circular equilibrium position.
The fluid domain $\Omega$ is the square $[0,1]^2$ and the initial position of
the structure is an ellipse centered at $(0.5,0.5)$ with the following initial
configuration 
\[
\X_0(s)= 
\begin{pmatrix}
0.5+0.25 \sqrt{2}\cos s\\
\displaystyle 0.5+\frac {0.25}{\sqrt{2}}\sin s
\end{pmatrix},
\quad s\in[0,2\pi].
\]
We used the following physical parameters $\rho^{\rm f} =\rs=1$, $\mu=1$.
Moreover, we assume that the structure is an elastic string with stiffness
$\kappa=2$. 

\subsection{Stability}
The purpose of this paragraph is to illustrate the stability results of Theorem~\ref{th:stab}. 
Unconditional stability is obtained for Algorithm~\ref{pb:split} with $r=1$  and conditional 
stability with $r=2$. We compare the results with those obtained with the strongly coupled scheme,  Algorithm \ref{pb:fullydiscr}, 
for which unconditional stability has been established in~\cite{BCG15}. 

The model problem consists in the evolution of an ellipsoidal structure toward
a circular equilibrium position. The only force that drives the motion is the
elastic reaction force to the initial deformation, hence we expect that the
energy of the system decreases to a plateau value. In order to check the stability properties of the schemes, we performed long
term simulations decreasing the time step while keeping fixed fluid and solid
meshes. The fluid mesh is made of  $40\times40$ square elements subdivided into two triangles, whereas  the reference configuration of the structure is divided into
$40$ subintervals.

Figure~\ref{fig:stab} resports the time evolution of the total energy 
of the fluid-structure system, namely,  
\[
\mathbf{E}_h^n:=\rf\|\u_h^n\|^2_{0,\Omega}
+\rs\|\dd_h^n\|^2_{0,\Sigma}+\|\d_h^n\|^2_{\rm s}.
\]
 The results of the tests
are in agreement with the theoretical analysis. We can appreciate energy
decreasing for all the used time steps for Algorithm \ref{pb:fullydiscr} (a) and for the
 Algorithm~\ref{pb:split} with $r=1$  (b), whereas we see instability
for  Algorithm~\ref{pb:split} with $r=2$ (c) when the time step is not
sufficiently small.

\begin{figure}[h!]
    \centering
    \subfigure[Algorithm \ref{pb:fullydiscr}.]{\includegraphics[scale=0.2]{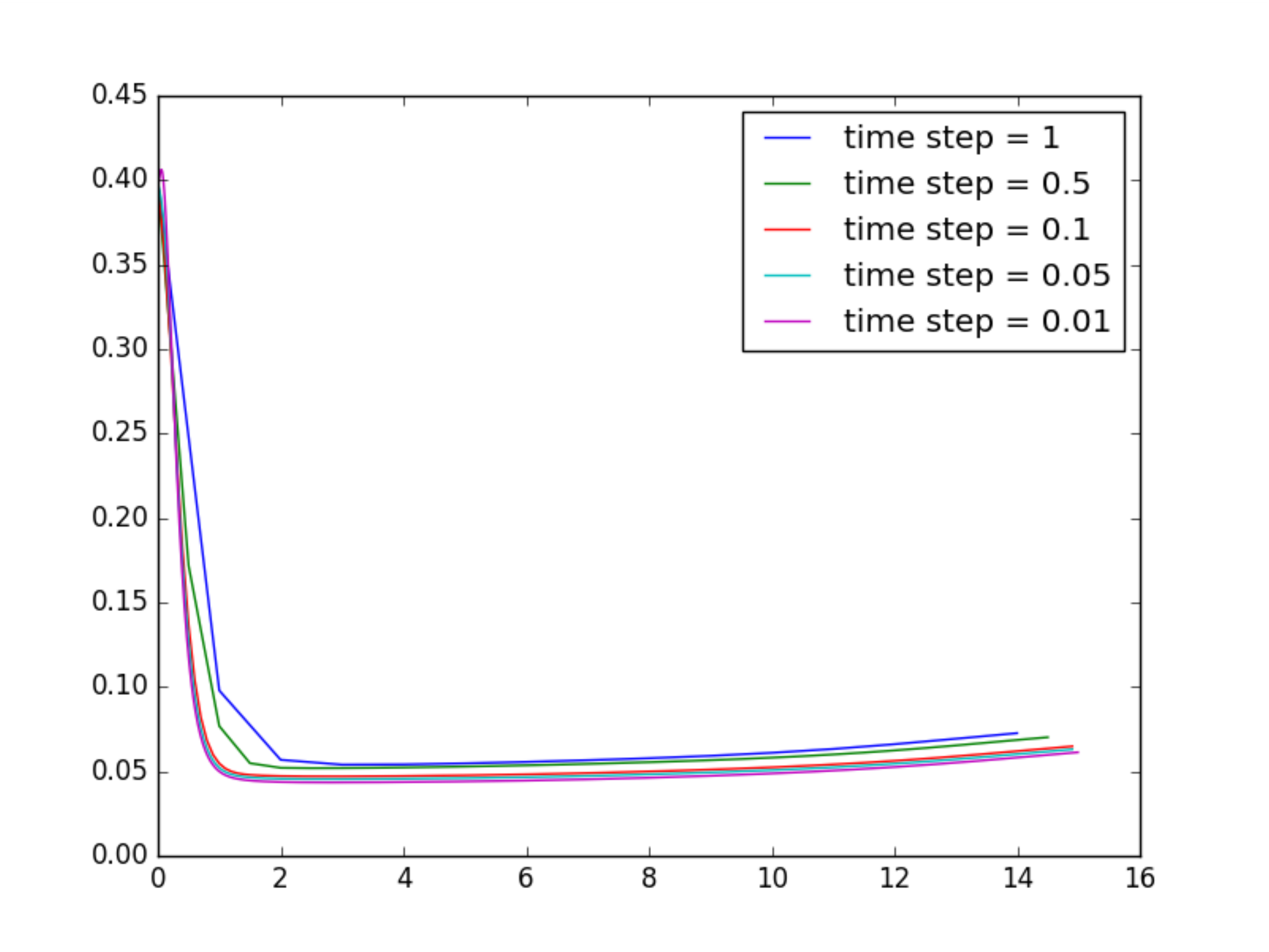}}       
    \subfigure[Algorithm~\ref{pb:split} with $r=1$.]{\includegraphics[scale=0.2]{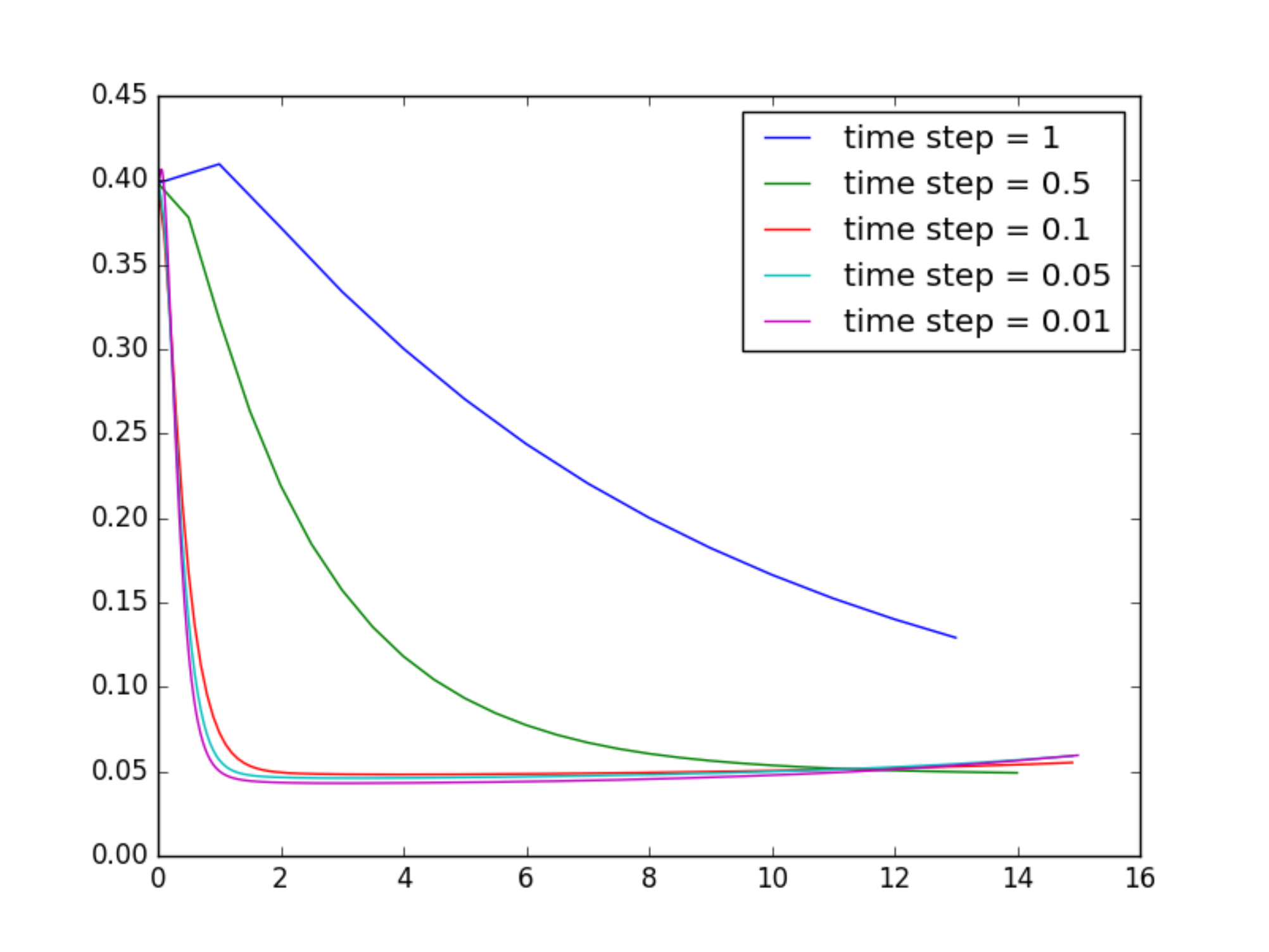}}
    \subfigure[Algorithm~\ref{pb:split} with $r=2$.]{\includegraphics[scale=0.2]{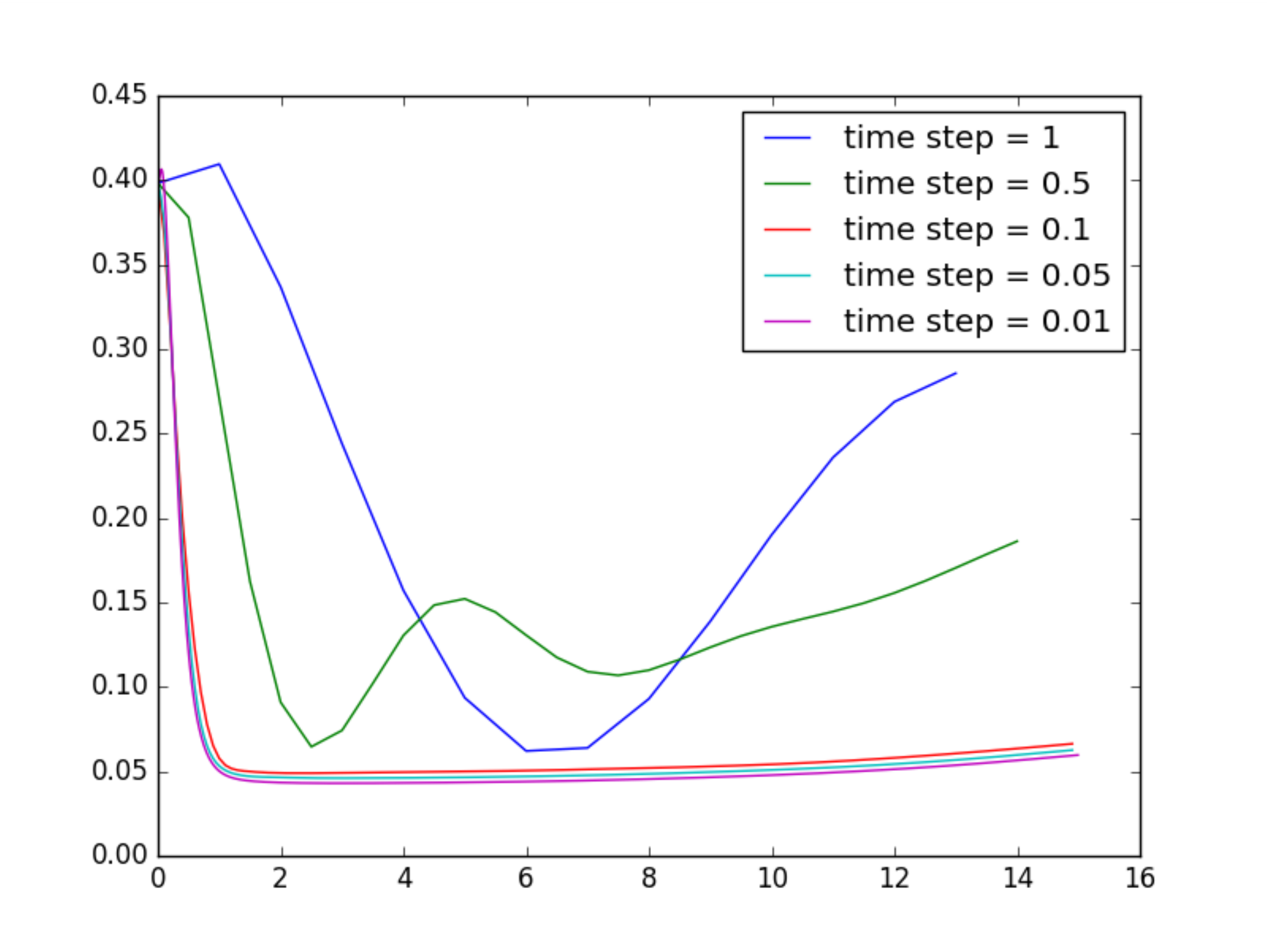}}
    \centering
    \caption{Evolution of the total energy $\mathbf{E}_h^n$ for different time-step lengths.}
\label{fig:stab}
\end{figure}

\subsection{Convergence}
In this  paragraph, we numerically investigate the convergence of Algorithms  \ref{pb:fullydiscr}  and \ref{pb:split}   
with respect to the mesh size and to the time-step length.  
We consider the same model problem as in the previous paragraph, with a
different initial configuration of the structure. More precisely, it
consists of the static equilibrium of a
circular elastic string, centered at the point
$(0.5,0.5)$ with radius $0.25$, and immersed in a fluid at rest. 

In order to check the convergence rate, we consider as reference solution the one obtained with Algorithm \ref{pb:fullydiscr} and 
the following discretization parameters:
\begin{align}\label{ref_parameters}
h_{\rm f}  =h_{\rm s}=\frac{1}{256}, \qquad\tau=5\cdot 10^{-5}.
\end{align}

Tables~\ref{tb:mono}--\ref{tb:part2} reports the spatial convergence history for Algorithms  \ref{pb:fullydiscr}  and \ref{pb:split},
respectively.  Here, the time-step length is fixed to $\dt=0.01$ and errors are evaluated at the final time $t=0.5$. 
The three schemes provide practically the same behavior and we observe a
sub-optimal rate, which is driven by the regularity of the solution.

\begin{table}[h!]
\caption{Algorithm \ref{pb:fullydiscr}. Spatial convergence for
$\tau = 0.01$.}
\renewcommand\arraystretch{1.2}
\begin{tabular}{|r|c|c|c|c|c|}
\hline
$h_{\rm f} =h_{\rm s} $ & 1/8 & \multicolumn{1}{c|}{1/16} & \multicolumn{1}{c|}{1/32} &
\multicolumn{1}{c|}{1/64} & \multicolumn{1}{c|}{1/128} \\
\hline
$\|\bold{u}_h^n-\bold{u} \|_{0,\Omega}$ & \multicolumn{1}{r|}{7.65E-3} &
5.92E-3 & 2.29E-3 & 8.56E-4 & 2.94E-4 \\
Rate & -- & 0.37 & 1.37 & 1.42 & 1.54 \\ \hline
$\|\bold{\dot{d}}_h^n-\bold{\dot{d}} \|_{0,\Sigma}$ &
\multicolumn{1}{r|}{5.43E-4} & 4.29E-4 & 2.23E-4 & 1.06E-4 & 5.93E-5 \\
Rate & -- & 0.34 & 0.94 & 1.07 & 0.84 \\ \hline
$\|\bold{d}_h^n-\bold{d} \|_{\rm s}$ & \multicolumn{1}{r|}{3E-02} & 1.58E-2
& 8.29E-3 & 4.69E-3 & 2.82E-3 \\
Rate & -- & 0.93 & 0.93 & 0.82 & 0.73 \\
\hline
\end{tabular}
\label{tb:mono}
\end{table}

\begin{table}[h!]
\caption{Algorithm~\ref{pb:split} with $r=1$. Spatial convergence for
$\tau = 0.01$.}
\renewcommand\arraystretch{1.2}
\begin{tabular}{|r|c|c|c|c|c|}
\hline
$h_{\rm f}=h_{\rm s}$ & 1/8 &  {1/16} & {1/32} &
 {1/64} &  {1/128} \\ \hline
$\|\bold{u}_h^n-\bold{u}\|_{0,\Omega}$ & \multicolumn{1}{r|}{7.61E-03} &
5.91E-3 & 2.28E-3 & 8.53E-4 & 2.91E-4 \\ 
Rate &  --  & 0.37 & 1.38 & 1.42 & 1.55 \\ \hline
$\|\bold{\dot{d}}_h^n-\bold{\dot{d}}\|_{0,\Sigma}$ &
\multicolumn{1}{r|}{5.17E-4} & 4.15E-4 & 2.19E-4 & 1.05E-4 & 5.91E-5 \\
Rate &  --  & 032 & 0.92 & 1.05 & 0.83 \\ \hline
$\|\bold{d}_h^n-\bold{d}\|_{\rm s}$ & \multicolumn{1}{r|}{2.99E-2} & 1.57E-2
& 8.28E-3 & 4.69E-3 & 2.82E-3 \\ 
Rate &  --  & 0.93 & 0.93 & 0.82 & 0.73 \\
\hline
\end{tabular}
\label{tb:part1}
\end{table}

\begin{table}[h!]
\caption{Algorithm~\ref{pb:split} with $r=2$. Spatial convergence for
$\tau = 0.01$.}
\renewcommand\arraystretch{1.2}
\begin{tabular}{|r|c|c|c|c|c|}
\hline
$h_{\rm f}=h_{\rm s}$ & 1/8 & {1/16} &  {1/32} &
 {1/64} &  {1/128} \\ \hline
$\|\bold{u}_h^n-\bold{u}\|_{0,\Omega}$ & \multicolumn{1}{r|}{7.60E-3} &
5.91E-3 & 2.28E-3 & 8.53E-4 & 2.93E-4 \\ 
Rate & --  & 0.36 & 1.38 & 1.42 & 1.54 \\ \hline
$\|\bold{\dot{d}}_h^n-\bold{\dot{d}}\|_{0,\Sigma}$ &
\multicolumn{1}{r|}{5.15E-4} & 4.16E-4 & 2.19E-4 & 1.06E-4 & 5.89E-5 \\
Rate &  -- & 0.31 & 0.93 & 1.05 & 0.84 \\ \hline
$\|\bold{d}_h^n-\bold{d}\|_{\rm s}$ & \multicolumn{1}{r|}{2.99E-2} & 1.57E-2
& 8.28E-3 & 4.69E-3 & 2.82E-3 \\ 
Rate & -- & 0.93 & 0.93 & 0.82 & 0.73 \\
\hline
\end{tabular}
\label{tb:part2}
\end{table}

We test now the convergence rate with respect to the time-step length $\dt$. We ran
tests with the following mesh sizes $h_f = h_{\rm s} = 1/64 $, varying
the time step as follows:
$
\tau \in \lbrace 1/2^i\rbrace_{i=4\ldots8}.
$
We compute the errors with respect to a reference solution obtained solving, for
each advancing scheme, the problem with $\tau_{ref} = 5E-05 s$ and
$h_{ref}^f = h_{ref}^s = 1/64 $.

We observe that the partitioned scheme with order two extrapolation results
to be stable for sufficiently small values of time step, hence we used values
of $\dt$ in the stability range. 
We can observe that the error of the partitioned scheme with order
two extrapolation, approaches the value of the monolithic error when the time
step reduces properly. This seems to be in agreement with the convergence
results in Theorem~\ref{th:error}.
As far as the order one partitioned scheme, we can see that the rates of
convergence appear to be higher. Actually, the error is much higher for big
time steps and it is close to the monolithic error for small ones.
All the errors have the same behavior as the time step goes to zero as the
theory predicts.

\begin{table}[h!]
\caption{Algorithm \ref{pb:fullydiscr}. Temporal convergence for $h_{\rm f}=h_{\rm s} = 1/64$.}
\setlength{\tabcolsep}{10pt}
\renewcommand\arraystretch{1.2}
\begin{tabular}{|r|c|c|c|c|c|c|}
\hline
$\tau$ & 1/16 & \multicolumn{1}{c|}{1/32} & {1/64} &
 {1/128} & {1/256} & {1/512} 
\\ \hline
$\|\bold{u}_h^n-\bold{u}\|_{0,\Omega}$ & \multicolumn{1}{r|}{2.65E-6} &
1.73E-6 & 1.07E-6 & 5.96E-7 & 3.13E-7 & 1.58E-7 
\\ 
Rate & -- & 0.61 & 0.69 & 0.84 & 0.93 & 0.98
\\ \hline
$\|\bold{\dot{d}}_h^n-\bold{\dot{d}}\|_{0,\Sigma}$ &
\multicolumn{1}{r|}{6.03E-6} & 4.07E-6 & 2.43E-6 & 1.32E-6 & 6.86E-7 &3.46E-7 
\\ 
Rate & -- & 0.57 & 0.74 & 0.88 & 0.95 & 0.99
\\\hline
$\|\bold{d}_h^n-\bold{d}\|_{\rm s}$ & \multicolumn{1}{r|}{4.44E-4} & 2.22E-4
& 1.11E-4 & 5.52E-5 & 2.74E-5 & 1.35E-5 
\\ 
Rate &  & 1.00 & 1.00 & 1.00 & 1.01 & 1.02 
\\ \hline
\end{tabular}
\label{tb:monot}
\end{table}

{\small
\begin{table}[h!]
\caption{Algorithm~\ref{pb:split} with $r=1$. Temporal convergence for $h_{\rm f}=h_{\rm s} = 1/64$.}
\setlength{\tabcolsep}{10pt}
\renewcommand\arraystretch{1.2}
\begin{tabular}{|r|c|c|c|c|c|c|}
\hline
$\tau$ & 1/16 & \multicolumn{1}{c|}{1/32} & {1/64} &
 {1/128} & {1/256}  & {1/512} %\multicolumn{1}{c|}{1/512}
 \\ \hline
$\|\bold{u}_h^n-\bold{u}\|_{0,\Omega}$ & \multicolumn{1}{r|}{2.40E-4} &
9.90E-5 & 3.08E-5 & 6.86E-6 & 1.57E-6  & 4.04E-7
 \\
Rate & -- & 1.28 & 1.69 & 2.17 & 2.12 & 1.96
\\ \hline
$\|\bold{\dot{d}}_h^n-\bold{\dot{d}}\|_{0,\Sigma}$ &
\multicolumn{1}{r|}{1.6E-4} & 4.36E-5 & 1.29E-5 & 3.63E-6 & 1.11E-6  & 4.02E-07 
\\ 
Rate &  & 1.87 & 1.75 & 1.84 & 1.71  & 1.46
\\ \hline
$\|\bold{d}_h^n-\bold{d}\|_{s}$ & \multicolumn{1}{r|}{1.81E-3} & 1.08E-3
& 4.37E-4 & 1.05E-4 & 3.33E-5   & 1.42E-5
 \\ 
Rate & -- & 0.75 & 1.30 & 2.06 & 1.65  & 1.23 
\\ \hline
\end{tabular}
\label{tb:part1t}
\end{table}
}

\begin{table}[h!]
\caption{Algorithm~\ref{pb:split} with $r=2$.  Temporal convergence for $h_{\rm f}=h_{\rm s} = 1/64$.}
\setlength{\tabcolsep}{10pt}
\renewcommand\arraystretch{1.2}
\begin{tabular}{|r|c|c|c|c|c|c|c|}
\hline
$\tau$ & 1/16 &  {1/32} & {1/64} &
{1/128} &  {1/256}  &{1/512} 
\\ \hline
$\|\bold{u}_h^n-\bold{u}\|_{0,\Omega}$ & \multicolumn{1}{r|}{2.21E-4} &
6.34E-5 & 4.64E-6 & 6.39E-7 & 3.17E-7 & 1.59E-7 
\\ 
Rate & -- & 1.81 & 3.77 & 2.86 & 1.01  & 0.99 
\\ \hline
$\|\bold{\dot{d}}_h^n-\bold{\dot{d}}\|_{0,\Sigma}$ &
\multicolumn{1}{r|}{8.32E-5} & 6.06E-5 & 6.04E-6 & 1.40E-06 & 6.83E-7 &3.40E-07 
\\ 
Rate & -- & 0.46 & 3.33 & 2.11 & 1.03 & 1.01
\\ \hline
$\|\bold{d}_h^n-\bold{d}\|_{\rm s}$ & \multicolumn{1}{r|}{1.20E-3} & 6.03E-4
& 1.26E-4 & 5.50E-5 & 2.73E-5 & 1.35E-05 
\\
Rate & -- & 0.98 & 2.25 & 1.20 & 1.01 & 1.02 
\\ \hline
\end{tabular}
\label{tb:part2t}
\end{table}
We report the numerical values of the error and the computed convergence
rates in Tables~\ref{tb:monot}, \ref{tb:part1t} and \ref{tb:part2t}.
All the schemes provide a rate of converges which is about 1 confirming the
theoretical results of Theorem~\ref{th:error}.

\subsection{Temporal accuracy} 

In order to illustrate  the accuracy of Algorithms  \ref{pb:fullydiscr}  and \ref{pb:split}, we show the evolution of two nodes on the structure during
\begin{figure}[h!]
\centering 
\includegraphics[scale=0.6]{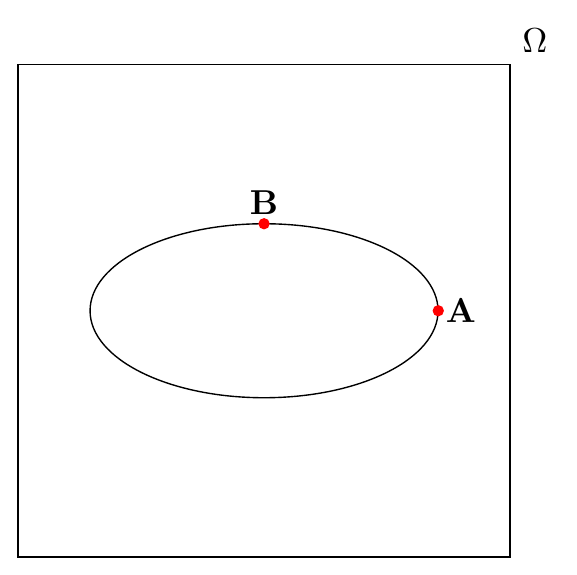}
\caption{Computational domain and interface with the control points $A$ and $B$.}
\label{fig:ellipse}
\end{figure}
the simulation relative to the ellipsoidal structure evolving to a circular
configuration.
At the beginning of the numerical test, the major and minor axes of the ellipse
are aligned with the abscissa en coordinate axes, respectively, see
Figure~\ref{fig:ellipse}.

\begin{figure}[h!]
\center
\subfigure[ $\dt = 0.1$.]
{\includegraphics[scale=0.2]{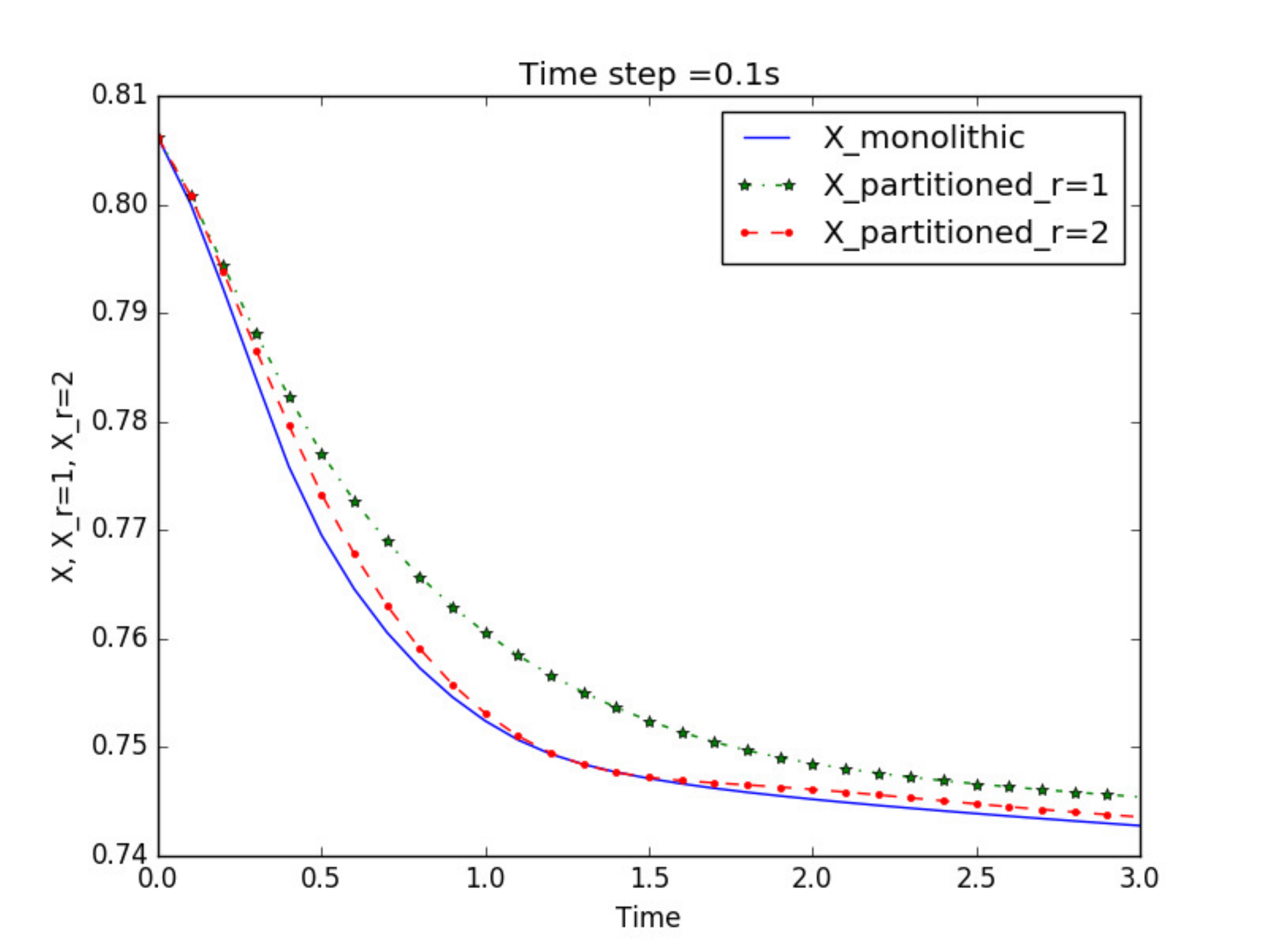}}
\subfigure[$\dt = 0.05$.]
{\includegraphics[scale=0.2]{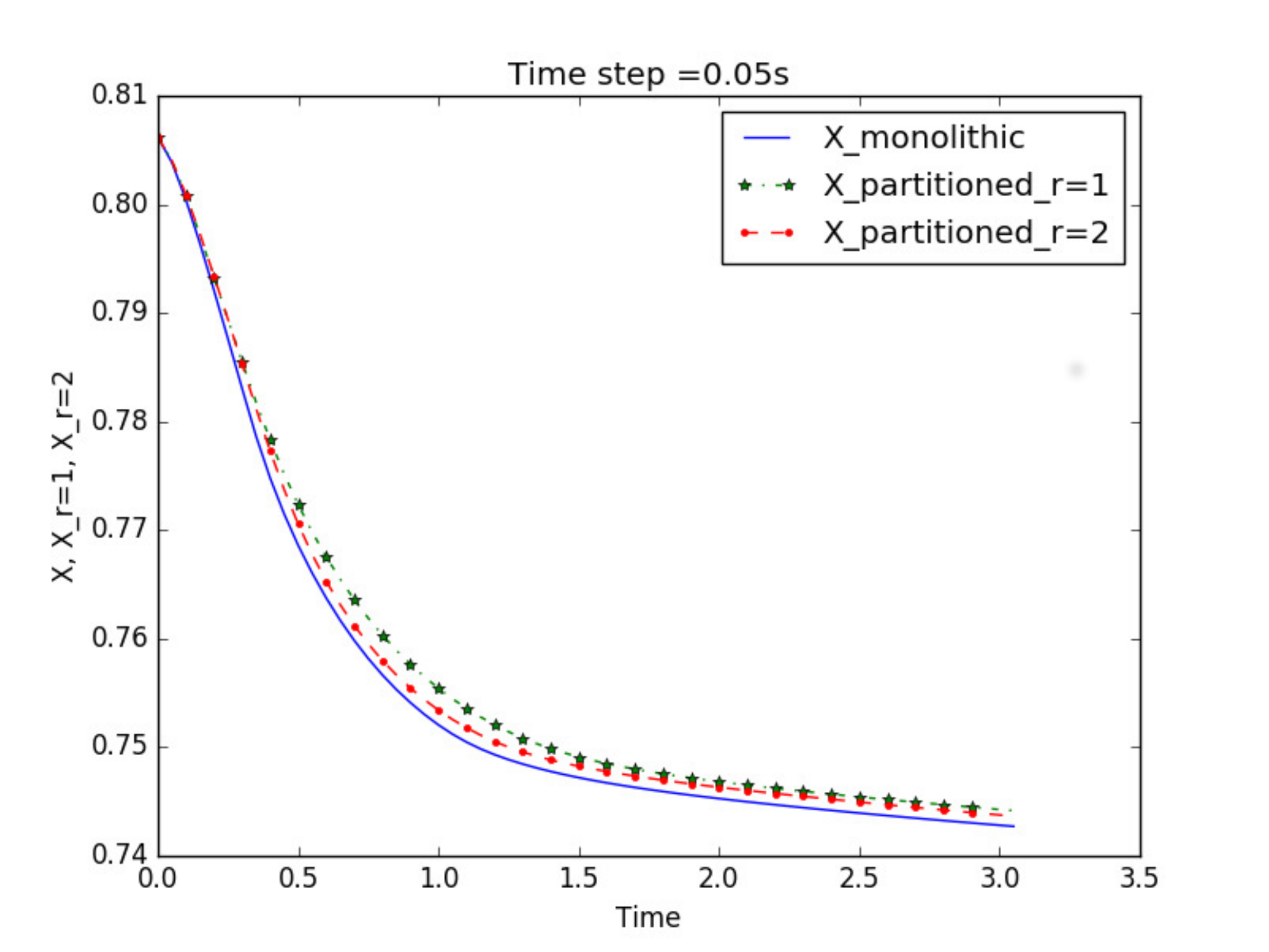}}
\subfigure[ $\dt = 0.01$.]
{\includegraphics[scale=0.2]{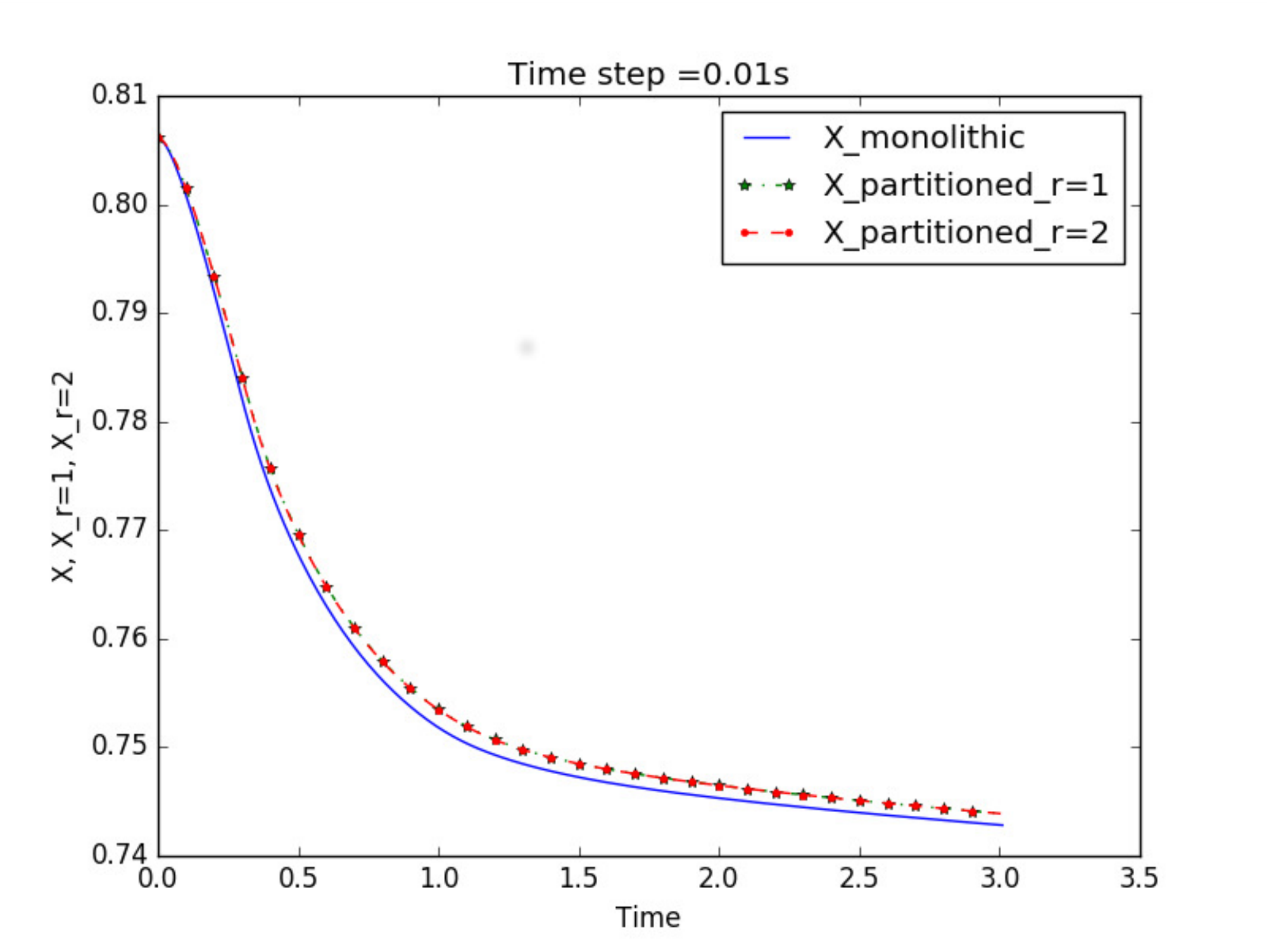}}
\caption{Evolution of the abscissa of point $A$ for different time-step lengths.}
\label{fig:evolA}
\end{figure}
\begin{figure}[h!]
\center
\subfigure[$\dt = 0.1$.]
{\includegraphics[scale=0.2]{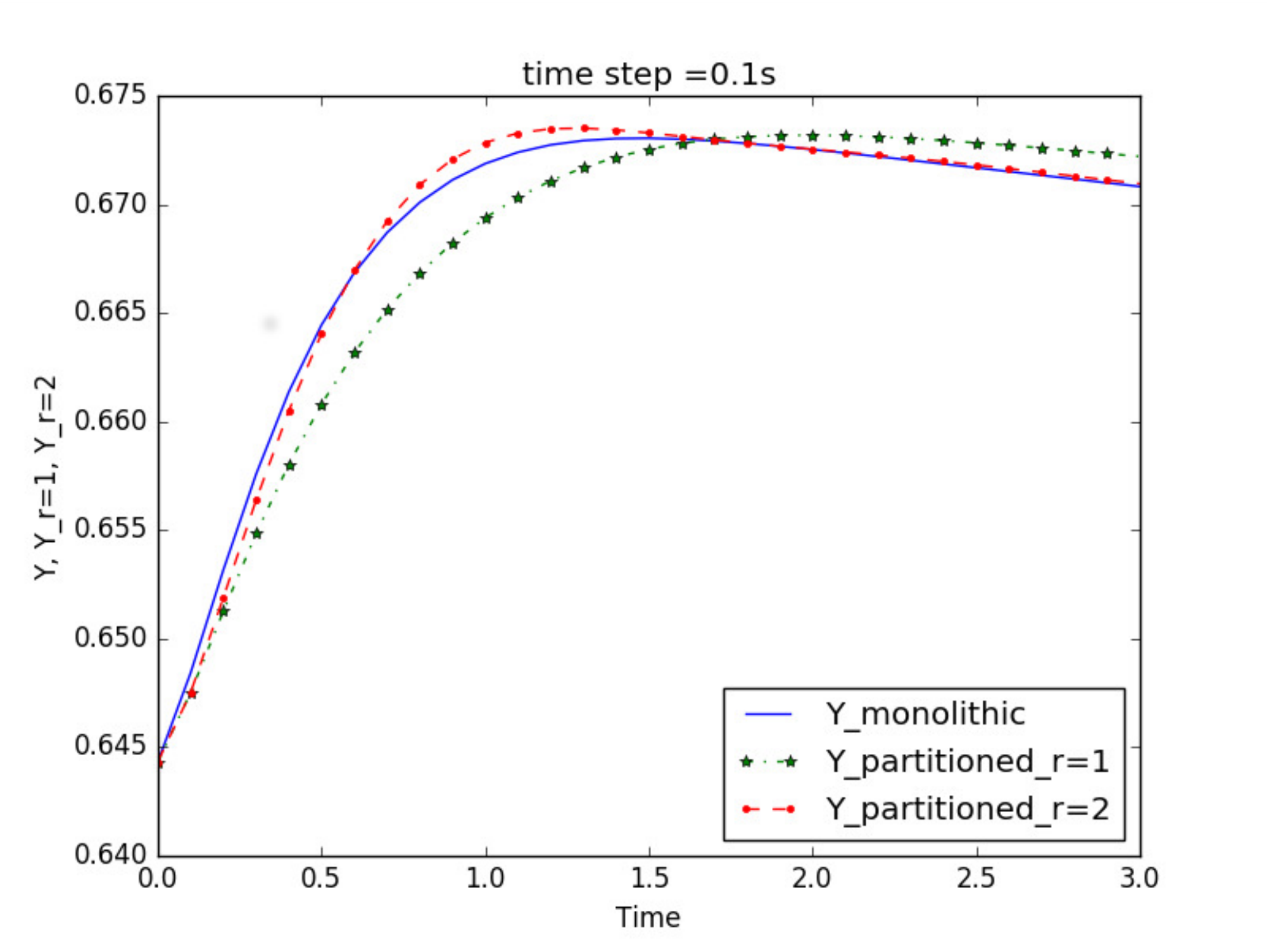}}
\subfigure[$\dt = 0.05$.]
{\includegraphics[scale=0.2]{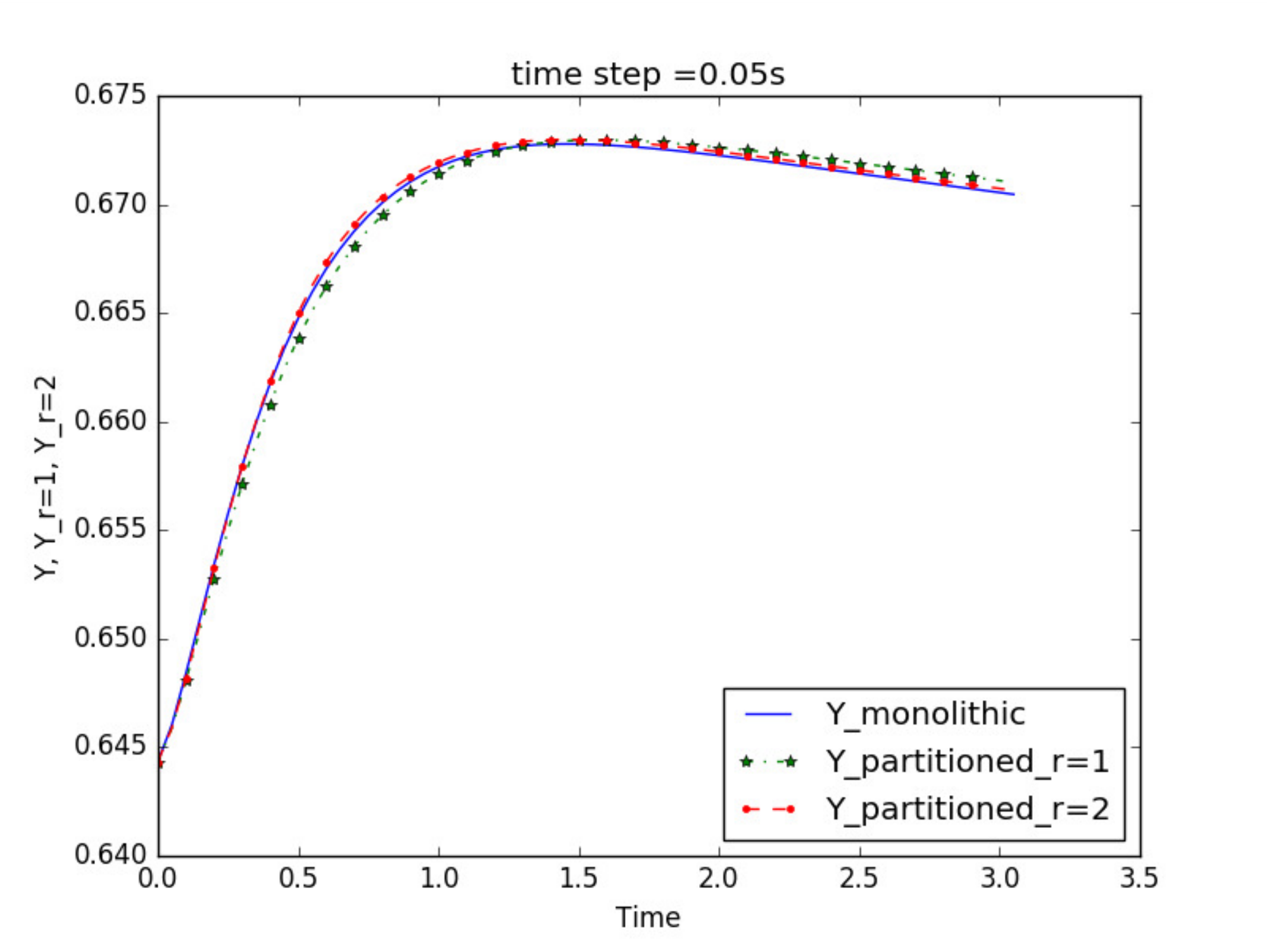}}
\subfigure[$\dt = 0.01$.]
{\includegraphics[scale=0.2]{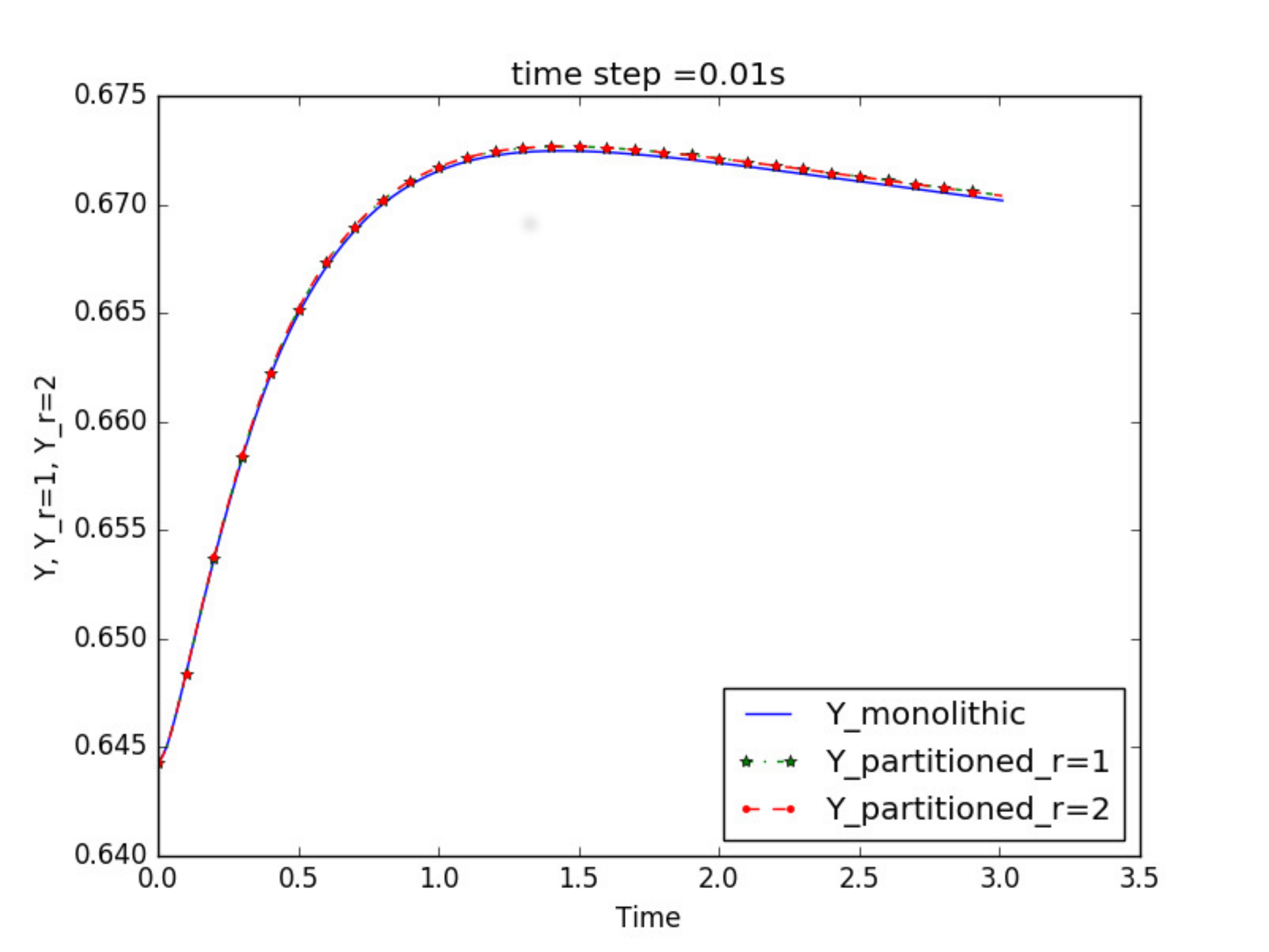}}
\caption{Evolution of the ordinate of point $B$ for different time-step lengths.}
\label{fig:evolB}
\end{figure}

Figures~\ref{fig:evolA} and \ref{fig:evolB}   show the evolutions of the abscissa of $A$ and of the ordinate of $B$, respectively, for different time-step 
lengths: $\dt=0.1,\, 0.05,\, 0.01$. The impact of the extrapolation order $r$ on the accuracy of Algorithm \ref{pb:split}  is clearly visible with the coarsest discretization. 
 Indeed, for $\dt=0.1$ we observe that the accuracy of Algorithm \ref{pb:split} with $r=2$ is superior to $r=1$. After time-step refinement,  $\dt=0.05,\, 0.01$, this difference 
 is  negligible and Algorithms \ref{pb:fullydiscr} and  \ref{pb:split} provide very close approximations. These numerical findings  are in agreement with relation  \eqref{eq:dn12}, 
 which shows that Algorithm  \ref{pb:split} can be seen as a kinematic perturbation of Algorithm \ref{pb:fullydiscr}. The size of this perturbation depends on both the extrapolation
  order $r$ and the time-step length $\dt$.

\section*{Acknowledgements}

The third author is member of the INdAM Research group GNCS and her research
is partially supported by IMATI/CNR and by PRIN/MIUR.

\bibliographystyle{plain}
\bibliography{paper}

\end{document}